\documentclass[reqno]{amsart}

\usepackage{amsmath}
\usepackage{amssymb}
\usepackage{mathrsfs}
\usepackage{xypic}

\theoremstyle{plain}
    \newtheorem{thm}{Theorem}[section]
    \newtheorem{ppn}[thm]{Proposition}
    \newtheorem{lem}[thm]{Lemma}
    \newtheorem{cor}[thm]{Corollary}
    
\theoremstyle{definition}
    
    \newtheorem{rmk}[thm]{Remark}
    \newtheorem{ex}[thm]{Example}

\numberwithin{equation}{section}

\allowdisplaybreaks

\def\Ker{\operatorname{Ker}}\def\Coker{\operatorname{Coker}}
\def\Hom{\operatorname{Hom}}\def\Spec{\operatorname{Spec}}
\def\C{\mathbb{C}}\def\Q{\mathbb{Q}}\def\R{\mathbb{R}}\def\Z{\mathbb{Z}}\def\F{\mathbb{F}}
\def\ol#1{\overline{#1}}\def\wt#1{\widetilde{#1}}\def\wh#1{\widehat{#1}}
\def\os#1#2{\overset{#1}{#2}}
\def\ot{\otimes}\def\op{\oplus}\def\ra{\rightarrow}
\def\a{\alpha}\def\b{\beta}\def\g{\gamma}\def\d{\delta}\def\pd{\partial}\def\e{\varepsilon}
\def\k{\kappa}\def\m{\mu}\def\z{\zeta}\def\vphi{\varphi}
\def\o{\omega}\def\O{\Omega}

\def\sO{\mathscr{O}}\def\sD{\mathscr{D}}\def\sM{\mathscr{M}}
\def\angle#1{{\langle #1 \rangle}}
\def\G{\Gamma}
\def\L{\Lambda}\def\s{\sigma}
\def\Im{\operatorname{Im}}
\def\Image{\operatorname{Image}}

\def\Re{\operatorname{Re}}
\def\1{\mathbf{1}}

\def\r{\rho}
\def\i{\iota}
\def\ve{\varepsilon}

\def\dR{\mathrm{dR}}

\def\rank{\operatorname{rank}}
\def\Ext{\operatorname{Ext}}

\def\MHS{\mathrm{MHS}}

\def\F#1#2#3{F\left({#1\atop #2};#3\right)}

\def\P{\mathbb{P}}
\def\sH{\mathscr{H}}\def\sE{\mathscr{E}}
\def\nb{\nabla}
\def\y{\eta}
\def\Res{\operatorname{Res}}
\def\floor#1{\lfloor#1\rfloor}
\def\ceil#1{\lceil#1\rceil}

\def\Gr{\operatorname{Gr}}
\def\D{\Delta}
\def\Per{\mathrm{Per}}

\def\cA{{\mathscr A}}
\def\cH{{\mathscr H}}
\def\cD{{\mathscr D}}

\def\cO{{\mathscr O}}

\def\hra{\hookrightarrow}
\def\lra{\longrightarrow}

\begin{document}

\title{CM periods, CM Regulators and Hypergeometric Functions, I}
\author{Masanori Asakura and Noriyuki Otsubo}
\address{Department of Mathematics, Hokkaido University, Sapporo, 060-0810 Japan}
\email{asakura@math.sci.hokudai.ac.jp}
\address{Department of Mathematics and Informatics, Chiba University, Chiba, 263-8522 Japan}
\email{otsubo@math.s.chiba-u.ac.jp}

\begin{abstract}
In this article, we prove the Gross--Deligne conjecture on CM periods for motives associated with $H^2$ of certain surfaces fibered over the projective line. We prove for the same motives some formula which expresses the $K_1$-regulators in terms of hypergeometric functions ${}_3F_2$, and the non-triviality of the regulators. 
\end{abstract}

\date{June 8, 2016}

\maketitle


\section{Introduction}

Periods and regulators of a motive over a number field are very important invariants, whose arithmetic significance can be seen from their conjectural relations with values of the $L$-function at integers. Such conjectures include those of Birch--Swinnerton-Dyer, Deligne, Bloch, Beilinson and Bloch--Kato. 
If the motive has complex multiplication (CM) by a number field, especially by an abelian field, those invariants take a special form. 

If $A$ is an abelian variety with CM by a subfield of the $N$th cyclotomic field, 
its periods are written in terms of values of the gamma function at $\frac{1}{N}\Z$. 
When $A$ is an elliptic curve, the formula is due to Lerch \cite{lerch} and was rediscovered by Chowla--Selberg \cite{chowla-selberg}. 
Gross \cite{gross} gave a geometric proof of a generalization of the formula and proposed a conjecture for any motivic de Rham--Hodge structure with CM by an abelian field, whose precise form was given by Deligne. 
Using Shimura's monomial relation \cite{shimura}, Anderson \cite{anderson} proved the formula for CM abelian varieties by reducing to the case of Fermat curves. 

In this paper, we study a surface $X$ fibered over $\P^1$ ($t$-line) with the general fiber defined by 
$$y^p=x^a(1-x)^b(t^l-x)^{p-b}$$ 
where $l$ and $p$ are distinct prime numbers. 
It admits an action of $\mu_{lp}$ and its second cohomology modulo the image of classes supported at singular fibers gives a de Rham--Hodge structure $H=(H_\dR, H_B)$ with multiplication by $K:=\Q(\mu_{pl})$ (see Sect. 2.2). 
We shall prove that $H_B$ is one-dimensional over $K$ (Theorem \ref{thm 0}). 
For each embedding $\chi \colon K \hookrightarrow \C$, let $H^\chi$ be the eigen-component.  
We shall determine its period and the Hodge type independently, 
and prove the Gross--Deligne conjecture. 

\begin{thm}[Period formula, see Theorem \ref{thm 1}]
For each $\chi \colon K \hookrightarrow \C$, let 
$\chi(\z_p)=\z_p^n$, $\chi(\z_l)=\z_l^m$ and put 
$\a=\{\frac{na}{p}\}$, $\b=\{\frac{nb}{p}\}$, $\m=\{\frac{m}{l}\}$. 
Then we have
$$\Per(H^\chi) \sim_{K'^\times} B(\b,\m)B(1-\b,\b-\a+\m)$$
where $K':=\Q(\mu_{2lp})$, 
and the Gross--Deligne conjecture holds. 
\end{thm}

On the other hand, regulators of the Fermat curve of degree $N$ are written in terms of values at $1$ of hypergeometric functions ${}_3F_2$ with parameters in $\frac{1}{N}\Z$ \cite{otsubo-1}. The conjectural relation with $L$-values is verified for some cases in \cite{otsubo-2}, \cite{otsubo-3}. 
Recall that the beta function is related with the value at $1$ of Gauss' hypergeometric function ${}_2F_1$. It is also suggestive that the classical polylogarithm can be written as
$$\mathrm{Li}_k(x)=x \cdot {}_{k+1}F_k\left({1,1, \dots, 1 \atop 2,\dots, 2}; x\right),$$
and hence special values of Dirichlet $L$-functions are written in terms of ${}_{k+1}F_k$-values. 

For the surface $X$, we consider the Beilinson regulator \cite{beilinson} from the motivic cohomology to the Deligne cohomology 
$$r_\sD \colon H^3_\sM(X,\Q(2)) \to H^3_\sD(X_\C,\Q(2)).$$
In terms of algebraic $K$-theory, $H^3_\sM(X,\Q(2))=(K_1(X) \ot_\Z \Q)^{(2)}$ (the second eigenspace for the Adams operations). 
Let $Z_1$ be the union of fibers over $\mu_l$ and consider the image of $H^3_{\sM,Z_1}(X,\Q(2))\to H^3_\sM(X,\Q(2))$. 
The Deligne cohomology can be regarded as functionals on $F^1H_\dR^2(X)$ up to periods, and we restrict them to $F^1H_\dR$. 

\begin{thm}[Regulator formula, see Theorem \ref{thm 2}]\label{thm intro 2}
Let $\chi$ be an embedding such that $H_\dR^\chi \subset F^1H_\dR$. 
Then, for any $z \in H^3_{\sM,Z_1}(X,\Q(2))$ and $\o \in H_\dR^\chi$, we have
$$r_\sD(z)(\o) \sim_{K^\times} B(1-\a,\b)\cdot  {}_3F_2\left({1-\a,\b,\b-\a+\m\atop1-\a+\b,\b-\a+\m+1};1\right)$$
where $\a$, $\b$, $\m$ are as before. 
\end{thm}

Moreover, we shall show the non-vanishing of the regulator image under a mild assumption (Theorem \ref{thm 3}). 

Regarding these examples, it is tempting to ask if the regulators and hence the $L$-values of a motive with CM by an abelian field can be written in terms of values of ${}_{k+1}F_k$, with $k$ depending on the weight. In a forthcoming paper \cite{a-o}, we shall study more general fibrations of varieties over $\P^1$ with multiplication by a number field whose relative $H^1$ has a special type of monodromy.  

Concerning the period conjecture, there is a result of Maillot--Roessler \cite{maillot-roessler} using Arakelov theory on the absolute value of the period. 
Recently, Fres\'an \cite{fresan} proved the formula for the alternating product of the determinants for any smooth projective variety with a finite order automorphism by reducing to a result of Saito--Terasoma \cite{saito-terasoma}. 
Since we prove $\dim_KH_B=1$ and $H^1(X)=H^3(X)=0$, 
the Gross--Deligne conjecture for our $H$ is a special case of Fres\'an's result. 
We need, however, our precise computations for the study of regulators. 

Our method is quite different from the previous works mentioned above. 
A crucial step is to compute explicitly Deligne's canonical extension $\sH_e$ of the Gauss--Manin connection on the relative first de Rham cohomology. 
Our fibration is smooth outside $D:=\{0,\infty\}\cup \mu_l$ and there is a connection
$$\nabla \colon \sH_e \to \O^1_{\P^1}(\log D) \ot \sH_e.$$ 
We shall describe it explicitly and determine the Hodge structure of $H$. 
The $1$-periods of the fiber are Gauss hypergeometric functions ${}_2F_1$. 
By the integral representation of Euler type, the $2$-periods of $X$ are firstly written in terms of ${}_3F_2$-values, which then turn out to be ${}_2F_1$-values. The conjecture follows by comparing these computations. 

It is more delicate in general to compute the regulators of given motivic elements, even for a fibration of curves. 
Here we use a technique of the first author \cite{asakura}, 
which we summarize in an appendix for the convenience of the reader.  
Via the canonical extension, we shall represent elements of $F^1H_\dR$ by certain rational $2$-forms.
Then the  regulators are expressed as integrals of those rational forms over Lefschetz thimbles, which are again written in terms of ${}_3F_2$-values. 

This paper proceeds as follows. 
In Sect. 2, we fix the setting and compute the $1$-periods of the fiber and $2$-periods of $X$. In Sect. 3, we determine the Gauss--Manin connection and the canonical extension. 
In Sect. 4, we determine the Hodge structure and show that $H_B$ is one-dimensional over $K$. In Sect. 5, we give a basis of $F^1H_\dR$ and verify the Gross--Deligne conjecture. 
In Sect. 6, we prove the regulator formula and discuss the non-vanishing.  
Appendix (Sect. 7) provides a short exposition of a technique developed in 
\cite{asakura}.
\subsection*{Acknowledgements}
This work started when the authors stayed at University of Toronto. We would like to thank heartily Kumar Murty for his hospitality. 
The second author would like to thank Bruno Kahn for helpful discussions. 
Finally, we would like to thank Spencer Bloch for valuable comments on an earlier version. 
This work is supported by JSPS Grant-in-Aid for Scientific Research 24540001, 25400007 and by Inamori Foundation.   

\subsection*{Notations}
Throughout this paper, $\ol\Q$ denotes the algebraic closure of $\Q$ in $\C$. For each positive integer $N$, $\mu_N$ denotes the group of $N$th roots of unity and we put $\z_N=e^{2\pi i/N}$. 
For a real number $x$, we write
$x=\floor x + \{x\}$ with $\floor x \in \Z$, $0\le\{x\} <1$, and  put $\ceil{x}=-\floor{-x}$.  
For $\a \in \C$ and an integer $n\geq 0$, $(\a)_n=\prod_{i=0}^{n-1}(\a+i)$ is the Pochhammer symbol and the generalized hypergeometric function is defined by
$${}_pF_q\left({\a_1,\dots,\a_p \atop \b_1,\dots,\b_q};x\right)=\sum_{n=0}^\infty \frac{\prod_{i=1}^p(\a_i)_n}{\prod_{j=1}^q(\b_j)_n} \frac{x^n}{n!}.$$
We often drop the subscripts from ${}_pF_q$. It converges at $x=1$ when $\mathrm{Re}(\sum_j \b_j-\sum_i \a_i)>0$. 
We use the standard notation for the product of $\G$-values
$$\Gamma\left({\a_1,\dots,\a_p \atop \b_1,\dots, \b_q}\right)=\frac{\prod_{i=1}^p \Gamma(\a_i)}{\prod_{j=1}^q \Gamma(\b_j)}.$$
For a variety $X$ over $\ol\Q$, $H_\dR^n(X)=H_\dR^n(X/\ol\Q)$ denotes the algebraic de Rham cohomology and $H^n(X,\Q)$ denotes the Betti cohomology of the analytic manifold $X(\C)$, or the associated mixed Hodge structure. 

\section{Preliminaries}

\subsection{The setting}
Let $p$, $l$ be distinct prime numbers and $a$, $b$, $c$ be integers with $0<a,b,c<p$ (we shall soon assume that $b+c=p$). 
We define a fibration of curves $f \colon X \ra \P^1$ as follows. 
Let $g\colon Y \to \P^1$ be a proper flat morphism over $\ol\Q$ whose fiber $Y_t$ at $t \in \P^1$ is the normalization of the curve defined by 
$$y^p=x^a(1-x)^b(t-x)^c.$$
Then, $g$ is smooth outside $\{0,1,\infty\}$ and by the Riemann--Hurwitz formula, the genus of the generic fiber is $p-1$. The fiber $Y_1$ is a union of $\P^1$ intersecting transversally with each other. 
We have an automorphism $\s$ of order $p$ of $Y$ over $\P^1$ defined by
$$\s(x,y)=(x,\z_p^{-1}y).$$

Let  $g^{(l)}\colon Y^{(l)} \ra \P^1$ be the base change of $g$ by the morphism $\P^1 \to \P^1; t \mapsto t^l$. 
The action of $\s$ extends naturally to $Y^{(l)}$. On the other hand, the automorphism 
$$\tau(t)=\z_l t$$ 
of $\P^1$ induces an automorphism $\tau$ of $Y^{(l)}$ over $Y$. 
There is a desingularization $X$ of $Y^{(l)}$ such that $\s$ and $\tau$ extend to automorphisms of $X$ respectively over $\P^1$ and $Y$ (for example, if one takes a sequence of blow-ups
only at the singular points, then $\s$ and $\tau$ extend automatically). 
As a result, we obtain a fibration $f\colon X \ra \P^1$ of curves in the commutative diagram
$$\xymatrix{
X\ar[r] \ar[dr]_f & Y^{(l)} \ar[r] \ar[d]^{g^{(l)}} \ar@{}[dr]|\square & Y \ar[d]^g \\ & \P^1\ar[r] & \P^1 
}$$
and for $t \not\in\{0,\infty\}\cup \mu_l$, the fiber $X_t$ is isomorphic to $Y_{t^l}$.

\subsection{CM de Rham--Hodge structures}

A {\it de Rham--Hodge structure} is a quadruple $H=(H_\dR, H_B, \iota, F^\bullet)$ of
\begin{itemize}
\item 
a finite-dimensional $\ol\Q$-vector space $H_\dR$, 
\item a finite-dimensional $\Q$-vector space $H_B$, 
\item 
an isomorphism $\iota \colon H_\dR\ot_{\ol\Q}\C \to H_B\ot_\Q\C$, 
\item a descending filtration $F^\bullet H_\dR$ which induces a Hodge structure on $H_B$ via $\iota$. 
\end{itemize}
For a proper smooth variety $X$ over $\ol\Q$, its $n$th de Rham and Betti cohomology groups, the comparison isomorphism and the Hodge filtration define a de Rham--Hodge structure $H^n(X)$.  

Let $K$ be a finite extension of $\Q$.  
We say that $H$ admits a {\it $K$-multiplication} if  we are given $K$-actions on $H_\dR$ and $H_B$ which are  compatible with $\iota$ and $F^\bullet$. Moreover, we say that {\it $H$ has CM by $K$} if $\dim_KH_B=1$. 
For each embedding $\chi \colon K \hookrightarrow \C$, 
let $H_\dR^\chi$, $H_B^\chi:=(H_B\ot_\Q \ol\Q)^\chi$ denote the subspace on which $K$ acts as the multiplication via $\chi$. If $\dim_K H_B=1$, then these subspaces are $1$-dimensional over $\ol\Q$. 
Choosing any bases $\o_\dR \in H_\dR^\chi$ and $\o_B \in H_B^\chi$, 
we define the {\it period} $\Per(H^\chi) \in \C^\times$ by
$$\iota(\o_\dR)= \Per(H^\chi)\o_B.$$ 
By the ambiguity of the choices, 
$\Per(H^\chi)$ is only well-defined up to $\ol\Q^\times$. 
If $(H_\dR, F^\bullet)$ is already defined over $K$, the period is well-defined up to $K^\times$. 

Let $X$ be as in Sect. 2.1 and let 
$$Z = X \times_{\P^1} (\{0,\infty\}\cup \mu_l)$$ be 
the union of the bad fibers.  
Note that $Z$ is stable under the actions of $\s$ and $\tau$. 
Put
$$R=\Q[\s,\tau], \quad K=\Q(\mu_{lp})$$ 
and regard $K$ as an $R$-algebra by $\s \mapsto \z_p$, $\tau \mapsto \z_l$. 
The de Rham--Hodge structure we consider in this paper is 
$$H:= \Coker(H_Z^2(X) \to H^2(X)) \ot_R K.$$ 
It admits a $K$-multiplication and we shall show that $\rank_K H_B=1$ (Theorem \ref{thm 0}). 
An embedding $\chi\colon K \hookrightarrow \C$ is identified with $h \in (\Z/lp\Z)^\times$ such that $\chi(\z_{lp})=\z_{lp}^h$. 
If 
$$\Coker(H_Z^2(X) \to H^2(X)) = \bigoplus_{m \in \Z/l\Z,n\in\Z/p\Z} H^{(m,n)}$$
denotes the decomposition into the eigenspaces on which $\tau$ (resp. $\s$) acts by $\z_l^m$ (resp. $\z_p^n$), we have 
$$H=\bigoplus_{m\ne 0, n\ne 0} H^{(m,n)}.$$

\subsection{Periods of the fiber}

For $n=1,\dots, p-1$ and integers $i$, $j$, $k$, put a rational $1$-form on $Y_t$ by
$$\o_n^{ijk}=\frac{x^i(1-x)^j(t-x)^k}{y^n} dx.$$
Then, we have 
\begin{equation}\label{omega eigen}
\s^*\o_n^{ijk}=\z_p^n\o_n^{ijk}.
\end{equation}
Let $0<t<1$, and $\d_0$ be a path on $Y_t$ from $(0,0)$ to $(t,0)$ defined by 
$$x=ts, \ y=\sqrt[p]{x^a(1-x)^b(t-x)^c}. $$
Let $\d_1$ be a path on $Y_t$ from $(t,0)$ to $(1,0)$ defined by 
$$x=t+(1-t)s, \ y=\e^c \sqrt[p]{x^a(1-x)^b(x-t)^c}$$
where we put
$$\ve=\begin{cases}
i & \text{if $p=2$}, \\
-1 &  \text{if $p$ is odd}. 
\end{cases}$$

If we put
$$\k_m=(1-\s)_*\d_m \quad (m=0,1),$$
these define $1$-cycles on $Y_t$, and we have
\begin{equation}\label{kappa period}
\int_{\k_m} \o_n^{ijk}=\int_{\d_m} (1-\s)^*\o_n^{ijk}=(1-\z_p^n)\int_{\d_m}\o_n^{ijk}.
\end{equation}

\begin{lem}\label{period}
Fix integers $i, j, k \ge 0$. 
For $n=1,\dots, p-1$, put
$$\a=\frac{na}{p}-i, \quad \b=\frac{nb}{p}-j, \quad \g=\frac{nc}{p}-k.$$
Then we have
\begin{align*}
\int_{\d_0} \o_n^{ijk}&=B(1-\a,1-\g) \cdot t^{1-\a-\g}\F{1-\a,\b}{2-\a-\g}{t},\\
\int_{\d_1} \o_n^{ijk}&= \e^{p\g} B(1-\b,1-\g)\cdot (1-t)^{1-\b-\g}\F{\a,1-\b}{2-\b-\g}{1-t}. 
\end{align*}
\end{lem}

\begin{proof}
The first equality follows directly from Euler's integral representation of the Gauss hypergeometric function ${}_2F_1$: 
$$B(b,c-b)\cdot \F{a,b}{c}{t}=\int_0^1(1-tx)^{-a}x^{b-1}(1-x)^{c-b-1}\; dx$$
(let $a=\b$, $b=1-\a$, $c=1-\a-\g$). 
The second one follows from the same formula and the transformation formula
$$\F{a,c-b}{c}{1-\frac{1}{t}}=t^a \cdot \F{a,b}{c}{1-t}. $$
\end{proof}

\subsection{Cohomology of the fiber}

We have decompositions
\begin{align*}
&H^1(Y_t,\C) = \bigoplus_{n=1}^{p-1} H^1(Y_t,\C)^{(n)}, \\
&H_1(Y_t,\Q(\mu_p)) = \bigoplus_{n=1}^{p-1} H_1(Y_t,\Q(\mu_p))^{(n)},\end{align*}
where ${}^{(n)}$ denotes the subspace on which $\s^*$ (resp. $\s_*$) acts as the multiplication by $\z_p^n$. 
Note that $H^1(Y_t,\C)^{(0)}=0$ since $Y_t/\mu_p$ is a rational curve. 
The natural paring induces a non-degenerate pairing 
$$H^1(Y_t,\C)^{(n)} \ot H_1(Y_t,\Q(\z_p))^{(n)} \ra \C.$$
We shall give bases of these spaces under a certain assumption. 

\begin{lem}\label{holo cond}
Let $n=1,\dots, p-1$ and $i, j, k \ge 0$ be integers. 
\begin{enumerate}
\item If $p \nmid a+b+c$, then $\o_n^{ijk}$ is a differential form of the second kind. 
\item
Moreover, $\o_n^{ijk}$ is holomorphic if and only if
\begin{align*}
& i \ge \frac{na+1}{p}-1, \quad j \ge \frac{nb+1}{p}-1, \quad k \ge \frac{nc+1}{p}-1, \\ 
& i+j+k \le \frac{n(a+b+c)-1}{p}-1. 
\end{align*}
\end{enumerate}
\end{lem}

\begin{proof}
See \cite{archinard} (18) (but see loc. cit. (13) for the correct sign in the fourth inequality).
\end{proof}

From now on, we assume:
\begin{equation*}
b+c=p.
\end{equation*}
Then the condition $p \nmid a+b+c$ is automatically satisfied. 
By Lemma \ref{holo cond}, $\o^{ijk}_n$ is holomorphic if and only if 
$$i=\left\lceil  \frac{na+1}{p}\right\rceil-1, \quad j=\left\lceil  \frac{nb+1}{p}\right\rceil-1, \quad
k=\left\lceil  \frac{nc+1}{p}\right\rceil-1,$$
and we write this $\o_n^{ijk}$ simply as $\o_n$.  
The $\a$, $\b$, $\g$ in Lemma \ref{period} become
\begin{equation*}\label{alphabetagamma}
\a=\left\{ \frac{na}{p}\right\}, \quad \b=\left\{\frac{nb}{p}\right\}, \quad 
\g=\left\{\frac{nc}{p}\right\}=1-\b.
\end{equation*}
In particular, $0<\a,\b,\g<1$. 
Though these depend on $n$, we shall suppress it from the notation. 
By Lemma \ref{period}, we have
\begin{equation}\label{period omega}
\begin{split}
\int_{\d_0}\o_n&=B(1-\a,\b)\cdot t^{\b-\a}\F{1-\a,\b}{1-\a+\b}{t},\\
\int_{\d_1}\o_n&=-\ve^{p\b}B(1-\b,\b)\cdot \F{\a,1-\b}{1}{1-t}.
\end{split}\end{equation}

For each $n$, let $i,j,k$ be as above and put
$$\y_n=\o_n^{i,j+1,k}.$$
Then, $\b$ is replaced with $\b-1$ in Lemma \ref{period} and we obtain
\begin{equation}\label{period eta}
\begin{split}
\int_{\d_0}\y_n&=B(1-\a,\b)\cdot t^{\b-\a}\F{1-\a,\b-1}{1-\a+\b}{t},\\
\int_{\d_1}\y_n&=-\ve^{p\b}B(1-\b,\b)\cdot(1-\b) (1-t)\F{\a,2-\b}{2}{1-t}.
\end{split}\end{equation}
Here we used $B(2-\b,\b)=(1-\b)B(1-\b,\b)$. 

\begin{ppn}Let $n=1,\dots, p-1$ and $0<t<1$. Then, 
$\{\o_n,\y_n\}$ is a basis of $H^1(Y_t,\C)^{(n)}$. 
\end{ppn}

\begin{proof}
By \eqref{omega eigen}, \eqref{kappa period}, \eqref{period omega} and \eqref{period eta}, 
$\o_n$, $\y_n$ are non-trivial elements of  $H^1(Y_t,\C)^{(n)}$. Since $\o_n$ is holomorphic and $\y_n$ is not, they are linearly independent. Since $\dim H^1(Y_t,\C)=2(p-1)$, the proposition follows. 
\end{proof}

\begin{ppn}\label{kappa}
Let $n=1,\dots, p-1$ and $0<t<1$. 
\begin{enumerate}
\item The projections of $\k_0, \k_1$ form a basis of $H_1(Y_t,\Q(\mu_N))^{(n)}$. 
\item As a $\Q[\s]$-module, $H_1(Y_t,\Q)$ is generated by $\k_0$ and $\k_1$. 
\end{enumerate}
\end{ppn}

\begin{proof}
Put the period matrix
$$M_n(t)=\begin{pmatrix}
\int_{\k_0}\o_n & \int_{\k_0}\y_n\\
\int_{\k_1}\o_n & \int_{\k_1}\y_n
\end{pmatrix}.$$
It suffices to show that $\det M_n(t) \ne 0$. 
Since $\prod_{n=1}^{p-1} \det M_n(t)$ is constant, it coincides with its limit as $t \to 1$. 
Hence the proposition follows from the lemma below. 
\end{proof}

\begin{lem}
We have
$$\lim_{t \ra 1} \det M_n(t)= \ve^{p\b}(1-\z_p^n)^2 \cdot \frac{B(\b,1-\b)}{1-\a}.$$ 
\end{lem}

\begin{proof}
By \eqref{kappa period}, \eqref{period omega}, \eqref{period eta}, we have
\begin{align*}
\det M_n(t)=& -\ve^{p\b}(1-\z_p^n)^2B(1-\a,\b)B(1-\b,\b) t^{\b-\a}
\\& \times
\det 
\begin{pmatrix}
\F{1-\a,\b}{1-\a+\b}{t}&\F{1-\a,\b-1}{1-\a+\b}{t} \\
\F{\a,1-\b}{1}{1-t}&(1-\b)(1-t)\F{\a,2-\b}{2}{1-t}
\end{pmatrix}. 
\end{align*}
Firstly, we have
$$\lim_{t \ra 1} (1-t)\F{1-\a,\b}{1-\a+\b}{t}=0.$$
This follows from the transformation formula (cf.  \cite{bateman}, p. 74 (2))
\begin{align*}
\F{1-\a,\b}{1-\a+\b}{t}&=\frac{1}{B(1-\a,\b)} \sum_{n=0}^\infty \frac{(1-\a)_n(\b)_n}{(n!)^2}(k_n-\log(1-t))(1-t)^n, 
\\k_n&:=2 \psi(n+1)-\psi(1-\a+n)-\psi(\b+n)
\end{align*}
where $\psi(t)=\G'(t)/\G(t)$ is the digamma function. 
On the other hand, by Euler's formula, we have 
$$\F{1-\a,\b-1}{1-\a+\b}{1}=\G\left({1-\a+\b \atop 2-\a, \b}\right)=\frac{1}{(1-\a)B(1-\a,\b)}.$$
Hence the lemma follows. 
\end{proof}


\subsection{Periods of $X$}

Now we consider the fibration $f\colon X \ra \P^1$. Recall that $X_t \simeq Y_{t^l}$. 
By abuse of notation, for each $s=0,1$, let $\d_s$ (resp. $\k_s$) be the path (resp. loop) on $X_t$ which corresponds to the one on $Y_{t^l}$ defined in \S2.3. 
For each $s$, let $\D_s$ be the $2$-simplex obtained by sweeping $\d_s$ along $0 \le t \le 1$. 
Since $\d_s$ is vanishing as $t \to s$, 
the Lefschetz thimble $(1-\s)_*\D_s$ has boundary on the fiber $X_{1-s}$. 
We shall use $(1-\s)_*\D_1$ (resp. $(1-\s)_*\D_0$) to compute the periods (resp. regulators). 
Again by abuse of notation, let $\o_n$ denote the pull-back to $X$ of the rational $1$-form $\o_n$ on $Y$ defined in \S2.4. 
For $n=1,\dots, p-1$ and an integer $m$, define rational $2$-forms on $X$ by 
$$\o_{m,n} = t^m \frac{dt}{t}\wedge \o_n, \quad \y_{m,n} = t^m\frac{dt}{t} \wedge \y_n.$$
We have evidently, 
$$(\tau^i\s^j)^*\o_{m,n}=\z_l^{mi}\z_p^{nj} \o_{m,n}, \quad 
(\tau^i\s^j)^*\y_{m,n}=\z_l^{mi}\z_p^{nj} \y_{m,n}.$$

\begin{ppn}\label{2-period}
Let $n=1,\dots, p-1$ and $\a=\{\frac{na}{p}\}$, $\b=\{\frac{nb}{p}\}$ as before. 
For an integer $m$, put $\mu=m/l$. 
\begin{enumerate}
\item If $\m>\a-\b$, then we have
\begin{align*}
&\int_{\D_1} \o_{m,n}
=-\frac{\ve^{p\b}}{l} \cdot B(\b,\mu)B(1-\b,\b-\a+\mu), 
\\&\int_{\D_1} \y_{m,n}
=-\frac{\ve^{p\b}(1-\b)}{l(1-\a+\mu)}\cdot B(\b,\mu)B(1-\b,\b-\a+\mu).
\end{align*}
\item  We have 
\begin{align*}
&\int_{\D_0} \o_{m,n}=\frac{B(1-\a,\b)}{l(\b-\a+\mu)}\cdot \F{1-\a,\b,\b-\a+\mu}{1-\a+\b,\b-\a+\mu+1}{1},
\\&\int_{\D_0} \y_{m,n}
=\frac{B(1-\a,\b)}{l(\b-\a+\mu)}\cdot \F{1-\a,\b-1,\b-\a+\mu}{1-\a+\b,\b-\a+\mu+1}{1}.\end{align*}
\end{enumerate}
\end{ppn}

\begin{proof}
Recall the integral representation of ${}_3F_2$ (cf. \cite[(4.1.2)]{slater}):
$$\G\left({c, e-c \atop e}\right)\F{a,b,c}{d,e}{t}=\int_0^1\F{a,b}{d}{tx}x^{c-1}(1-x)^{e-c-1}\, dx.$$
By \eqref{period omega}, we have
\begin{align*}
\int_{\D_1} \o_{m,n}
&=-\ve^{p\b}B(\b,1-\b) \int_0^1 \F{\a,1-\b}{1}{1-t^l}t^{m-1} \,dt
\\&=-\ve^{p\b}\frac{B(\b,1-\b)}{l}\int_0^1 \F{\a,1-\b}{1}{1-t}t^{\mu-1} \,dt
\\&=-\ve^{p\b}\frac{B(\b,1-\b)}{l}\int_0^1\F{\a,1-\b}{1}{t}(1-t)^{\mu-1}\, dt
\\&=-\ve^{p\b}\frac{B(\b,1-\b)}{l\mu}\F{\a,1-\b,1}{1,\mu+1}{1}
\\&=-\ve^{p\b}\frac{B(\b,1-\b)}{l\mu}\F{\a,1-\b}{\mu+1}{1}, 
\end{align*}
which converges by the assumption. 
Using Euler's formula 
$$\F{a,b}{c}{1}=\G\left({c,c-a-b\atop c-a,c-b}\right) \quad (\Re(c-a-b)>0)$$
and the functional equations
$$\G(x+1)=x\G(x), \ B(x,y)=\G\left({x,y\atop x+y}\right),$$
we obtain the first equality of (i). 
The others follow similarly, using \eqref{period eta} for $\y_{m,n}$. 
\end{proof}

\section{Canonical Extension}

In this section, we compute the Gauss--Manin connection of the fibration and determine its canonical extension to $\P^1$. 

\subsection{Gauss--Manin connection}

Let us start with the fibration $g \colon Y \to \P^1$; for a while, $t$ denotes the coordinate of the base scheme of $g$. 
Put
$$T=\P^1\setminus\{0,1,\infty\}, \quad Y_T=Y \times_{\P^1} T.$$
Then the restriction $g\colon Y_T \to T$ is smooth. 
Put
$$\sH=R^1g_*\O^\bullet_{Y_T/T}, \quad \O_T^1=\O^1_{T/\ol\Q},$$
and let
$$\nabla \colon \sH \to \O_T^1 \ot \sH$$
be the Gauss--Manin connection. 
For each $n=1,\dots, p-1$, let $\sH^{(n)} \subset \sH$ be the subbundle 
on which $\s^*$ acts as the multiplication by $\z_p^n$. 
Then $\sH^{(n)}$ is locally generated by $\o_n$, $\y_n$ as defined in Sect. 2.4, and the Hodge filtration $F^1\sH^{(n)}$ is generated by $\o_n$. 

\begin{ppn}\label{gauss-manin-g}
For $n=1,\dots, p-1$, the Gauss--Manin connection 
$$\nb \colon \sH^{(n)} \ra \O_T^1 \ot \sH^{(n)}$$
is given by
$$(\nb\o_n,\nb\y_n)=\frac{dt}{t}\ot(\o_n,\y_n)
\begin{pmatrix}1-\b & 0\\0 & 1-\a\end{pmatrix}
\begin{pmatrix}-1 & -1\\(1-t)^{-1} &1\end{pmatrix}$$
where we put $\a=\{\frac{na}{p}\}$, $\b=\{\frac{nb}{p}\}$ as before. 
\end{ppn}

\begin{proof}
We use the following standard derivation relations among Gauss hypergeometric functions (cf. \cite{slater}, (1.4.1.1), (1.4.1.6)):
\begin{align}
& \frac{d}{dt} \F{a,b}{c}{t}= \frac{ab}{c}\F{a+1,b+1}{c+1}{t},\label{c+1}\\
& \frac{d}{dt}\left(t^{c-1}\F{a,b}{c}{t}\right)=(c-1)t^{c-2}\F{a,b}{c-1}{t}. \label{c-1}
\end{align}
We also use the following contiguous relations (cf. loc. cit. (1.4.1), (1.4.3), (1.4.5), (1.4.9), (1.4.13)): 
\begin{align}
&(c-2a+(a-b)t)F+a(1-t)F[a+1]=(c-a)F[a-1],\label{1.4.1}\\
&(c-a-b)F+a(1-t)F[a+1]=(c-b)F[b-1],\label{1.4.3}\\
&(c-a-1)F+aF[a+1]=(c-1)F[c-1],\label{1.4.5}\\
&(a-1+(1+b-c)t)F+(c-a)F[a-1]=(c-1)(1-t)F[c-1],\label{1.4.9}\\
&c(1-t)F+(c-a)tF[c+1]=cF[b-1]\label{1.4.13}.
\end{align}
Here, $F=\F{a,b}{c}{t}$ and the notation $F[a+1]$ for example means $\F{a+1,b}{c}{t}$. 

We are reduced to show:
\begin{equation}\label{DM_n}
t\frac{d}{dt} M_n(t) = M_n(t) \begin{pmatrix}
1-\b & 0\\
0 & 1-\a
\end{pmatrix}
\begin{pmatrix}
-1 & -1\\
(1-t)^{-1} & 1
\end{pmatrix}.\end{equation}
We prove this for each row vector. 
For the first row vector, put
$$(f(t),g(t))=\left(
t^{\b-\a}\F{1-\a,\b}{1-\a+\b}{t}, 
t^{\b-\a}\F{1-\a,\b-1}{1-\a+\b}{t}\right). 
$$
Firstly, consider the case $\a \neq \b$. By \eqref{c-1}, we have
$$t\frac{d}{dt}(f(t),g(t))=\left((\b-\a)t^{\b-\a}\F{1-\a,\b}{-\a+\b}{t},
(\b-\a)t^{\b-\a}\F{1-\a,\b-1}{-\a+\b}{t}\right).$$
Applying \eqref{1.4.9} to $\F{\b,1-\a}{1-\a+\b}{t}$, we obtain
$$t\frac{d}{dt}f(t)=-(1-\b)f(t)+(1-\a)(1-t)^{-1} g(t).$$
Applying \eqref{1.4.5} to $\F{\b-1,1-\a}{1-\a+\b}{t}$, we obtain
$$t\frac{d}{dt}g(t)=-(1-\b)f(t)+(1-\a)g(t).$$
Hence we are done. 
Now consider the case $\a=\b$. Then 
$$(f(t),g(t)) =
\left(\F{1-\a,\a}{1}{t},
\F{1-\a,\a-1}{1}{t}\right).$$
By \eqref{c+1}, we have
$$\frac{d}{dt}(f(t),g(t))=
\left((1-\a)\a\F{2-\a,1+\a}{2}{t}, 
-(1-\a)^2\F{2-\a,\a}{2}{t}\right).$$
Applying \eqref{1.4.13} to $\F{2-\a,1+\a}{1}{t}$, we have
\begin{equation}\label{pf3-1-1}
t\frac{d}{dt}f(t)=\a(1-t)\F{2-\a,1+\a}{1}{t}-\a \F{2-\a,\a}{1}{t}.
\end{equation}
Applying \eqref{1.4.3} to $\F{1-\a,1+\a}{1}{t}$, we have
\begin{equation}
(1-\a)(1-t)\F{2-\a,1+\a}{1}{t}=\F{1-\a,1+\a}{1}{t}-\a f(t).
\end{equation}
Applying \eqref{1.4.1} to $\F{\a,1-\a}{1}{t}$, we have
\begin{equation}\label{pf3-1-0}
\a (1-t) \F{1-\a,1+\a}{1}{t}=(2 \a-1)(1-t)f(t)+(1-\a)g(t).
\end{equation}
Applying \eqref{1.4.3} to $\F{1-\a,\a}{1}{t}$, we have
\begin{equation}\label{pf3-1-4}
(1-t)\F{2-\a,\a}{1}{t}=g(t).
\end{equation}
Combining \eqref{pf3-1-1}--\eqref{pf3-1-4}, we obtain
$$t\frac{d}{dt}f(t)=(1-\a)\left(-f(t)+(1-t)^{-1}g(t)\right).$$
Applying \eqref{1.4.13} to $\F{\a,2-\a}{1}{t}$, we have
\begin{align*}
t\frac{d}{dt}g(t)&=(1-\a)\left(-\F{1-\a,\a}{1}{t}+(1-t)\F{2-\a,\a}{1}{t}\right)
\\&\os{\eqref{pf3-1-0}}{=}(1-\a)(-f(t)+g(t)).
\end{align*}
In both cases $\a \ne \b$ and $\a=\b$, we have proved \eqref{DM_n} for the first row vector. 

For the second row vector, put
$$(u(t),v(t))=\left(\F{\a,1-\b}{1}{1-t}, (1-\b)(1-t)\F{\a,2-\b}{2}{1-t}\right).$$
Then, by \eqref{c+1} and \eqref{c-1}, we have
$$\frac{d}{dt}(u(t),v(t))=-(1-\b)\left(\a\F{\a+1,2-\b}{2}{1-t}, \F{\a,2-\b}{1}{1-t}\right).$$
Applying \eqref{1.4.13} to $\F{\a,2-\b}{1}{1-t}$, we obtain
\begin{equation}\label{pf3-1-5}
t\frac{d}{dt}v(t)=-(1-\b)u(t)+(1-\a)v(t).
\end{equation}
Applying \eqref{1.4.3} to $\F{\a,2-\b}{2}{1-t}$, we have
\begin{equation}\label{pf3-1-6}
t\frac{d}{dt}u(t)=(\b-\a)(1-t)^{-1}v(t)-(1-\b)\b\cdot \F{\a,1-\b}{2}{1-t}.
\end{equation}
Applying \eqref{1.4.9} to $\F{2-\b,\a}{2}{1-t}$, we have
\begin{align}
(1-\b)\b\cdot\F{\a,1-\b}{2}{1-t}
&=\left(-(1-\b)(1-t)^{-1}+1-\a\right)v(t)-t \frac{d}{dt}v(t)\notag
\\&\os{\eqref{pf3-1-5}}=(1-\b)\left(u(t)-(1-t)^{-1}v(t)\right).\label{pf3-1-7}
\end{align}
Combining \eqref{pf3-1-6} and \eqref{pf3-1-7}, we obtain
$$t\frac{d}{dt}u(t)=-(1-\b)u(t)+(1-\a)(1-t)^{-1}v(t).$$
Hence we have proved \eqref{DM_n} for the second row vector. 
\end{proof}

\subsection{Canonical extension} 

Now, we return to the fibration $f \colon X \to \P^1$, and from now on, $t$ denotes the coordinate of the base scheme of $f$. Put 
$$D=\{0,\infty\}\cup\mu_l, \ T=\P^1 \setminus D, \ U=X\times_{\P^1} T, \ \sH=R^1f_*\O^\bullet_{U/T},$$ 
and let $\nb \colon \sH \ra \O_T^1 \ot \sH$ be the Gauss--Manin connection. 
The following is immediate from Proposition \ref{gauss-manin-g}. 

\begin{ppn}\label{gauss-manin} 
For $n=1,\dots, p-1$, the Gauss--Manin connection 
$$\nb \colon \sH^{(n)} \ra \O_T^1 \ot \sH^{(n)}$$
is given by
\begin{align*}
(\nb\o_n,\nb\y_n)
&=l \frac{dt}{t} \ot (\o_n,\y_n)\begin{pmatrix}1-\b&0\\0&1-\a\end{pmatrix}
\begin{pmatrix} -1& -1\\ \frac{1}{1-t^l} & 1 \end{pmatrix}\\
&=l \frac{ds}{s}  \ot (\o_n,\y_n)\begin{pmatrix}1-\b&0\\0&1-\a\end{pmatrix}
\begin{pmatrix} 1& 1\\ \frac{s^l}{1-s^l} & -1 \end{pmatrix}
\end{align*}
where $s=1/t$. 
\end{ppn}

Let $j\colon T \to \P^1$ denote the embedding. 
Write $\O^1=\O^1_{\P^1/\ol\Q}$ and let $\O^1(\log D)$ be the sheaf of differentials with logarithmic poles along $D$. 
Then the {\it canonical extension} 
$$\nabla \colon \sH_e \to \O^1(\log D) \ot \sH_e$$
of Deligne \cite[5.1]{deligne-book} is defined to be the unique sub-bundle of $j_*\sH$ satisfying the following properties: 
\begin{enumerate}
\item $\nb(\sH_e) \subset \O^1(\log D) \ot \sH_e$, 
\item For each $t\in D$, all the eigenvalues of $\Res_t(\nabla)$ lie in the interval $[0,1)$, 
where $\Res_t(\nabla)$ denotes the residue at $t$ of the connection matrix. 
\end{enumerate}
In fact, we have $\sH_e = R^1f_*\O_{X/\P^1}^\bullet(\log Z)$ by Steenbrink \cite[(2.18), (2.20)]{steenbrink}. 
This is determined as follows. 

\begin{ppn}\label{can ext}
For $n=1,\dots, p-1$, local bases of $\sH_e^{(n)}$ at $t\in D$ are given as follows: 
\begin{align*}
&\sH_e^{(n)}|_{0}=\begin{cases}
\left\langle \o_n-\y_n, t^{\ceil{(\a-\b)l}}((1-\b)\o_n-(1-\a)\y_n)\right\rangle & \text{if $\a \neq \b$},\\
\langle \o_n, \y_n\rangle & \text{if $\a=\b$}, 
\end{cases}
\\& \sH_e^{(n)}|_{\infty}=
\begin{cases}
\left\langle
t^{\floor{(1-\b)l}} ((1-\a-\b)\o_n+(1-\a)t^{-l}\y_n), t^{\floor{\a l}-l}\y_n)\right\rangle & \text{if $\a+\b \ne 1$}, \\
\left\langle t^{\floor{\a l}}\o_n, t^{\floor{\a l}-l}\y_n \right\rangle & \text{if $\a+\b = 1$},
\end{cases}
\\& \sH_e^{(n)}|_{\z}=\langle \o_n,\y_n\rangle \quad (\z \in \mu_l). 
\end{align*}
The residue matrices with respect to these bases are:
\begin{align*}
&\Res_{0}(\nb)=\begin{cases}
\begin{pmatrix}0&0\\0&\{(\b-\a)l\}\end{pmatrix} & \text{if $\a \neq \b$}, \\
l(1-\a)\begin{pmatrix}-1&-1\\1&1\end{pmatrix} & \text{if $\a=\b$}, 
\end{cases}
\\&\Res_{\infty}(\nb)=
\begin{cases}
\begin{pmatrix} \{(1-\b) l\} & 0 \\ 0 & \{\a l\}\end{pmatrix} & \text{if $\a+\b\ne 1$}, \\
\begin{pmatrix} \{\a l\}& 0 \\ (\a-1)l & \{\a l\} \end{pmatrix}
& \text{if $\a+\b= 1$}, \end{cases}
\\&\Res_{\z}(\nb)=-(1-\a)\begin{pmatrix}0 & 0 \\ 1& 0\end{pmatrix}. 
\end{align*}
\end{ppn}

\begin{proof}
Let $A$ be the matrix of connection from Proposition \ref{gauss-manin}. 
For each $t\in D$, we shall find a matrix $P$ with coefficients in local sections of $j_*\sO_U$, such that 
$(\o_n,\y_n)P$ is a local basis of $\sH_e$ at $t$. 
The connection matrix with respect to this basis is given by the Gauge transformation
$$A_P:=P^{-1}AP+P^{-1}P'$$
where $P'=\frac{d}{dt}P$. 
For $t=0$, we let
$$P=\begin{pmatrix}1 & 1-\b \\ -1 & -(1-\a)\end{pmatrix}
\begin{pmatrix}1&0\\0&t^{\ceil{(\a-\b)l}}\end{pmatrix}$$
if $\a \neq \b$, and $P=I$ (the unit matrix) if $\a =\b$. 
For $t=\z \in \mu_l$, we let $P=I$. 
Finally for $t=\infty$, we let 
$$P=\begin{pmatrix}1&0\\0&t^{-l}\end{pmatrix}
\begin{pmatrix}1-\a-\b&0\\1-\a&1\end{pmatrix}
\begin{pmatrix}t^{\floor{(1-\b)l}}&0\\0& t^{\floor{\a l}}\end{pmatrix}$$
if $\a+\b\ne 1$, and 
$$P=\begin{pmatrix}t^{\floor{\a l}}&0\\0& t^{\floor{\a l}-l}\end{pmatrix}$$
if $\a+\b=1$. 
Then, one verifies that $A_P$ satisfies the desired properties and its residue is given as stated.  
\end{proof}

To see the Hodge filtration, we rewrite the above bases as follows. 
\begin{cor}\label{local hodge fil}
Let $n=1,\dots, p-1$. 
\begin{align*}
&\sH_e^{(n)}|_{t=0}=\begin{cases}
\left\langle \o_n, t^{-\lfloor(\b-\a)l\rfloor}((1-\b)\o_n-(1-\a)\y_n)\right\rangle& \text{if $\a\le\b$},\\
\left\langle t^{\lceil(\a-\b)l\rceil}\o_n, \o_n-\y_n\right\rangle & \text{if $\a>\b$}.
\end{cases}
\\&\sH_e^{(n)}|_{t=\infty}=
\begin{cases}
\left\langle{
t^{\floor{(1-\b)l}}\o_n, t^{\floor{\a l}-l}\y_n)}
\right\rangle &\text{if $\floor{\a l} \ge \floor{(1-\b) l}$}, \\
\left\langle
t^{\lfloor \a l \rfloor}\o_n, 
t^{\floor{(1-\b)l}}((1-\a-\b)\o_n+(1-\a)t^{-l}\y_n)
\right\rangle &\text{if $\floor{\a l}<\floor{(1-\b) l}$}. 
\end{cases}
\\& \sH_e^{(n)}|_{t=\z}=\langle \o_n,\y_n\rangle \quad (\z \in \mu_l). 
\end{align*}
\end{cor}

Write $\sO=\sO_{\P^1}$ and define $F^1\sH_e=\sH_e \cap j_*(F^1\sH)$. 
Then, we have immediately: 
\begin{cor}\label{can ext fil}
Let $n=1,\dots, p-1$. 
\begin{enumerate}
\item We have $F^1\sH^{(n)}_e=\sO(i)t^j\o_n$ with
\begin{align*}
(i,j)=
\begin{cases}
(\floor{(1-\b)l},0) & \text{if $\floor{\a l} \ge \floor{(1-\b) l}$, $\a\le\b$},\\
(\floor{(1-\b)l}-\ceil{(\a-\b)l}, \ceil{(\a-\b)l}) & \text{if $\floor{\a l} \ge \floor{(1-\b) l}$, $\a>\b$},\\
(\floor{\a l},0) & \text{if $\floor{\a l}<\floor{(1-\b) l}$, $\a\le\b$},\\
(\floor{\a l}-\ceil{(\a-\b)l}, \ceil{(\a-\b)l}) & \text{if $\floor{\a l}<\floor{(1-\b) l}$, $\a>\b$}.
\end{cases}
\end{align*}
\item
According as the four cases as above, we have 
$$\Gr_F^0\sH_e^{(n)}=
\begin{cases}
\sO(-\ceil{(1-\a)l}+\floor{(\b-\a)l}) t^{-\lfloor(\b-\a)l\rfloor}((1-\b)\o_n-(1-\a)\y_n),\\
\sO(-\ceil{(1-\a)l})(\o_n-\y_n),\\
\sO(\floor{(\b-\a)l}-\ceil{\b l}) 
t^{-\floor{(\b-\a)l}}\left((1-\a-\b)t^l \o_n-(1-\b)\o_n+(1-\a) \y_n\right),\\
\sO(-\ceil{\b l})\left((1-\a-\b)t^l \o_n-(1-\a) (\o_n-\y_n)\right).
\end{cases}$$
Here, by abuse of notation, the images of $\o_n$, $\y_n$ in $\Gr_F^1\sH^{(n)}_e$ are written by the same letters. 
\end{enumerate}
\end{cor}

\begin{cor}\label{totdeg}
For each $\z \in \mu_l$, $X_\z$ is a normal crossing divisor in $X$ with rational irreducible components.  
\end{cor}

\begin{proof}By Proposition \ref{can ext}, the local monodromy of $H^1(X_t,\Q)$ at $t=\z$ is unipotent, 
hence $X_\z$ is normal crossing (see \cite[Theorem 1]{takeshi}). 
By the Clemens--Schmid exact sequence (see \cite{morrison}), $H^1(X_\z,\Q)$ is the kernel of 
the log local monodromy $N \colon H^1(X_t,\Q) \to H^1(X_t,\Q)$. 
Since $\rank N = \frac{1}{2} \dim H^1(X_t,\Q)$ by Proposition \ref{can ext}, $H^1(X_\z)$ is of pure weight $0$.  
Hence all the irreducible components are rational. 
\end{proof}

\section{Hodge Numbers}

In this section, we determine the Hodge numbers of the eigen-components of our $H$ and prove that it has CM by $K$, i.e. $\dim_K H_B=1$. 

\subsection{Localization sequence}

Let the notations be as in Sect. 3.2 and put $Z=X\setminus U$. 
We have the localization sequence
$$H^2_Z(X) \to H^2(X) \to H^2(U) \to H^3_Z(X) \to H^3(X)$$
both for the de Rham and Betti cohomologies. Let $\angle Z$ denote the image of the first map. Recall that we defined in Sect. 2.2 the de Rham--Hodge structure 
$$H=H^2(X)/\angle Z \ot_R K.$$

\begin{ppn}\label{H^1=0}
$H^1(X)=H^3(X)=0$. 
\end{ppn}

\begin{proof}
By Poincar\'e duality, it suffices to show $H^1(X,\Q)=0$. 
Since $H^1(X,\Q) \hookrightarrow W_1H^1(U,\Q)$, where $W_\bullet$ denotes the weight filtration, it suffices to show the vanishing of the latter. 
By the Leray spectral sequence, we have an exact sequence
$$0 \to H^1(T,\Q) \to H^1(U,\Q) \to H^0(T,R^1f_*\Q) \to 0.$$
By  the computation of $\Res_\infty(\nabla)$ in Proposition \ref{can ext}, 
for $n=1,\dots, p-1$, 
the local monodromy around $t=\infty$ of $H^1(X_t,\C)^{(n)}$ does not have $1$ as an eigenvalue. Hence we have $H^0(T,R^1f_*\Q)=0$  (recall that $H^1(X_t,\C)^{(0)}=0$). 
Since $H^1(T,\Q)$ is of weight $2$, we have $W_1H^1(U,\Q)=0$. 
\end{proof}

As a result, we have an exact sequence on the de Rham side (see \cite{hartshorne})
$$0 \to H^2_\dR(X)/\angle Z \to H^2_\dR(U) \os{\pd}\to H_1^\dR(Z) \to 0.$$
The middle term is described by the canonical extension as follows. 
The Leray spectral sequence yields an exact sequence
$$0 \ra H^1(T,\sH) \ra H_\dR^2(U) \ra H^0(T,R^2f_{*}\O_{U/T}^\bullet) \ra 0.$$
Since $\s^*$ acts on $R^2f_*\O^\bullet_{U/T}$ trivially, we have 
$H^1(T,\sH^{(n)}) \simeq H_\dR^2(U)^{(n)}$ for $n=1,\dots, p-1$. 
Put a complex of sheaves on $\P^1$ as
$$\sE=[\sH_e \os{\nb}{\ra} \O^1(\log D) \ot \sH_e].$$
Then, the map of complexes
$$\xymatrix{
\sH_e \ar[r] \ar[d] & \O^1(\log D) \ot \sH_e \ar[d]\\
j_*\sH  \ar[r] & j_*(\O^1_T\ot \sH)
}$$
induces an isomorphism
$$H^1(\P^1,\sE) \simeq H^1(T,\sH),$$
and the first group carries a mixed Hodge structure whose Hodge filtration is given as follows (see \cite{s-z}): 
\begin{equation}\label{fil}
\begin{split}
&F^0H^1(\P^1,\sE) =H^1(\P^1,\sE),\\
&F^1H^1(\P^1,\sE)=H^1(\P^1,F^1\sH_e \to \O^1(\log D)\ot \sH_e) ,\\
&F^2H^1(\P^1,\sE)=H^0(\P^1,\O^1(\log D)\ot F^1\sH_e). 
\end{split}\end{equation}
It follows: 
\begin{equation}\label{gr}
\begin{split}
&\Gr^0_FH^1(\P^1,\sE)=H^1(\P^1,\Gr_F^0\sH_e), 
\\&\Gr^1_FH^1(\P^1,\sE)
=\Coker\left(H^0(\P^1,F^1\sH_e)\os{\ol\nb}\ra H^0(\P^1,\O^1(\log D)\ot \Gr^0_F\sH_e)\right)
\end{split}
\end{equation}
where $\ol\nb$ is the map induced from the composition of $\nb$ and the projection $\sH_e \ra \Gr_F^0\sH_e$.  

\subsection{Residues}

For each $t\in D$, let
$$\pd_t \colon H_\dR^2(U)\to H_1^\dR(X_t)$$
be the $t$-component of the coboundary map $\pd$. 
Let $N_{t} \subset \sH_{e,t}$ be the image of the composite
$$\G(U_t,\sH_e) \os\nb\ra \G(U_t,\O^1(\log t) \ot \sH_e) \os{\Res_t}\ra \sH_{e,t}$$
where $U_t$ is a small open neighborhood of $t$. 
Then, it is not difficult to show that the diagram
$$\xymatrix{
H^1(\P^1,\sE) \ar[r]^\subset \ar[d]^{\Res_t} &  H^2_\dR(U) \ar[d]^{\pd_t} \\
\sH_{e,t}/N_{t} \ar[r]^\simeq & H_1^\dR(X_t) 
}$$
commutes where the lower map is an isomorphism. 
The following is immediate from Proposition \ref{can ext}. 
\begin{ppn}\label{N}
For $n=1,\dots, p-1$,  we have
\begin{align*}
&N_0^{(n)}=\bigl\langle t^{\ceil{(\a-\b)l}}((1-\b)\o_n-(1-\a)\y_n)\bigr\rangle, 
\\&N_\infty^{(n)}=\sH_{e,\infty},
\\&N_\z^{(n)}=\angle{\y_n}\quad (\z\in\mu_l).
\end{align*}
Therefore, we have
$$\dim H_1^\dR(X_t)^{(n)} = \begin{cases} 1 & \text{if $t=0$ or $t \in \mu_l$}, \\
0 & \text{if $t=\infty$}.
\end{cases}$$
\end{ppn}

Later, we shall use the following. 
\begin{lem}\label{res 0}
Let $n=1,\dots, p-1$. 
\begin{enumerate}
\item
If $\a \le \b$, then $t^m\o_n|_{t=0} \in N_0^{(n)}$ if $m>0$, and $\not\in N_0^{(n)}$ if $m=0$. 
\item 
If $\a>\b$, then $t^m\o_n|_{t=0}\in N_0^{(n)}$ if $m\ge \ceil{(\a-\b)l}$.  
\end{enumerate}
\end{lem}

\begin{proof}
By Corollary \ref{local hodge fil} and Proposition \ref{N}, this is trivial except when $\a>\b$ and $m=\ceil{(\a-\b)l}$.  In this case, we have
$$t^m\o_n|_{t=0}=t^m\o_n|_0+\frac{1-\a}{\a-\b} t^m(\o_n-\y_n)|_{t=0}
=\frac{t^m((1-\b)\o_n-(1-\a)\y_n)|_{t=0}}{\a-\b}\in N_0^{(n)}.$$
\end{proof}


\subsection{Hodge numbers} 

For each $n=1,\dots, p-1$, we obtained an exact sequence 
\begin{equation}\label{key sequence}
0 \to (H^2_\dR(X)/\angle{Z})^{(n)} \ra H^1(\P^1,\sE^{(n)}) \os{\Res}\ra 
 \sH_{e,0}^{(n)}/N_0^{(n)} \oplus \bigoplus_{\z\in\mu_l}\sH_{e,\z}^{(n)}/N_\z^{(n)}
 \to 0.
\end{equation}

Firstly, we give a basis of $F^2$. 
By \eqref{fil}, we have an embedding
$$\i\colon F^2(H_\dR^2(X)/\angle{Z})^{(n)} \hra \G(\P^1, \O^1(\log D) \ot F^1\sH_e^{(n)}).$$
By this, we identify $F^2(H_\dR^2(X)/\angle{Z})^{(n)}$ with the elements of the right member having trivial residues. Recall the rational $2$-forms 
$$\o_{m,n} = t^m \frac{dt}{t} \ot \o_n.$$ 

\begin{ppn}\label{F^2 basis}
For each $n=1,\dots, p-1$, a basis of $F^2(H_\dR^2(X)/\angle{Z})^{(n)}$ is given by 
$\{\o_{m,n} \mid m \in I^2_n\}$
where
$$I^2_n:=\bigl\{m \mid \max\{1, \ceil{(\a-\b)l} \} \leq m \leq \min\{\floor{\a l}, \floor{(1-\b) l}\}\bigr\}.$$
In particular, 
$$\dim F^2(H^2_\dR(X)/\angle{Z})^{(n)}=\min\{\floor{\a l}, \floor{(1-\b) l}\}-\max\{0,\floor{(\a-\b)l}\}.$$
\end{ppn}

\begin{proof}
Let $F^1\sH_e^{(n)} = \sO(i)t^j \o_n$ be as in Corollary \ref{can ext fil} (i). 
One easily sees that a basis of $H^0(\P^1,\O^1(\log D)\ot F^1\sH_e^{(n)})$ is given by
$$\o_{m,n}\ (j\leq m\leq i+j), \quad t^j\frac{dt}{t-\z}\ot \o_n\ (\z\in\mu_l).$$
For the first type, the residues at $\z \in \mu_l$ are trivial. 
By Lemma \ref{res 0}, $\Res_0(\o_{m,n})=t^m\o_n$ is trivial for $m\ge j$ 
unless $\a \le \b$ and $m=0$.  
For the second type, it has trivial residues except at $\z$ and 
$$\Res_\z\left(t^j\frac{dt}{t-\z}\ot \o_n\right)=t^j\o_n,$$
which is non-trivial by Proposition \ref{N}. 
These show that a basis of $F^2(H_\dR^2(X)/\angle{Z})^{(n)}$ is given by $\o_{m,n}$ with $j \le m \le i+j$ and $m \ne j=0$ if $\a\le \b$. 
Hence the proposition follows by Corollary \ref{can ext fil} (i). 
\end{proof}

Since $(\sH_{e,0}/N_0)^{(n)}$ and $(\sH_{e,\z}/N_\z)^{(n)}$ are all $1$-dimensional, 
the above proof implies the following. 

\begin{cor}\label{res F^2}
For $n=1,\dots, p-1$, we have
$$\Res\left(F^2H^1(\P^1,\sE^{(n)})\right)=
\begin{cases}
(\sH_{e,0}/N_0)^{(n)} \oplus \bigoplus_{\z \in \mu_l}\limits (\sH_{e,\z}/N_\z)^{(n)} & \text{if $\a \le \b$},\\
\bigoplus_{\z \in \mu_l}\limits (\sH_{e,\z}/N_\z)^{(n)} & \text{if $\a>\b$}.
\end{cases}
$$
\end{cor}

\begin{cor}\label{F^2 non-trivial}
Suppose that $p<l$. Then we have
$F^2(H_\dR^2(X)/\angle{Z})^{(n)}\ne 0$
for any $n=1,\dots, p-1$. 
\end{cor}

\begin{proof}
Since $\a, 1-\b \ge 1/p$, we have $l\a, l(1-\b)>1$. Since $\b \ge 1/p$ and $\a \le 1-1/p$, we have $(\a-\b)l < \a l-1, (1-\b)l-1$. Hence we have $I_n^2 \ne \emptyset$.
\end{proof}

Now, we determine the other Hodge numbers. 
\begin{lem}\label{res gr^1}
Let $n=1,\dots, p-1$. 
\begin{enumerate}
\item
If $\a \le \b$, then we have 
$$\Gr^1_F(H_\dR^2(X)/\angle{Z})^{(n)}= \Gr_F^1 H^1(\P^1,\sE^{(n)}).$$
\item
If $\a>\b$, then we have an exact sequence 
$$0 \ra \Gr^1_F(H_\dR^2(X)/\angle{Z})^{(n)} \ra \Gr_F^1 H^1(\P^1,\sE^{(n)}) \os{\Res_0}\ra 
(\sH_{e,0}/N_0)^{(n)}\ra 0.$$
\end{enumerate}
\end{lem}

\begin{proof}
By \eqref{key sequence} and Corollary \ref{res F^2}, we are left to show the non-triviality of $\Res_0$ in the case (ii). 
If $\floor{\a l}\ge\floor{(1-\b)l}$, consider
$$\frac{dt}{t(1-t^l)}\ot(\o_n-\y_n).$$
By Corollary \ref{can ext fil} (ii), this is an element of $H^0(\P^1,\O^1(\log D)\ot \Gr_F^0 \sH_e^{(n)})$. 
Its residue at $0$ is 
$\o_n-\y_n \not\equiv 0\pmod{N_0}$ by Proposition \ref{N}. 
If $\floor{\a l}<\floor{(1-\b)l}$,  consider similarly
$$\frac{dt}{t(1-t^l)}\ot \left((1-\a-\b)t^l\o_n-(1-\a)(\o-\y_n)\right),$$
whose residue at $0$ is $-(1-\a)(\o_n-\y_n)\not\equiv 0\pmod{N_0}$. 
\end{proof}

\begin{ppn}\label{hodge number 1}
For each $n=1,\dots, p-1$, we have
$$\dim \Gr_F^1(H_\dR^2(X)/\angle{Z})^{(n)}=\bigl|\floor{\a l}-\floor{(1-\b)l}\bigr|+\floor{|\a-\b|l}.$$
\end{ppn}

\begin{proof}
First, we show that the map
$$\ol\nb \colon H^0(\P^1,F^1\sH_e^{(n)}) \ra H^0(\P^1,\O^1(\log D)\ot \Gr^0_F\sH_e^{(n)})$$
is injective. 
Let $F^1\sH_e^{(n)}=\sO(i)t^j\o_n$ as in Corollary \ref{can ext fil} (i). 
Then, $H^0(\P^1,F^1\sH_e^{(n)})$ has a basis
$\{\o_{m,n} \mid j\le m\le i+j\}$, and 
\begin{align*}
\nb\o_{m,n}=\frac{dt}{t}t^m\left\{(m-l(1-\b))\o_n+\frac{l(1-\a)}{1-t^l}\y_n\right\} 
 \equiv l(1-\a)\frac{dt}{t(1-t^l)}t^m \y_n\not\equiv 0
\end{align*}
modulo $H^0(\P^1,\O^1(\log D)\ot F^1\sH_e^{(n)})$. 
Since $0 \le i<l$ in every case, $\o_{m,n}$ belong to different eigenspaces with respect to the $\tau$-action. Hence the non-vanishing implies the injectivity. 

By Corollary \ref{can ext fil} (ii), we have $\Gr_F^0\sH_e^{(n)} \simeq \sO(k)$ where 
$$k:=\begin{cases}
-\ceil{(1-\a)l}+\floor{(\b-\a)l}&  \text{if $\floor{\a l} \ge \floor{(1-\b)l}, \a\le \b$},\\
-\ceil{(1-\a)l} & \text{if $\floor{\a l} \ge \floor{(1-\b)l}, \a>\b$},\\
\floor{(\b-\a)l}-\ceil{\b l} & \text{if $\floor{\a l} < \floor{(1-\b)l}, \a\le \b$},\\
-\ceil{\b l} & \text{if $\floor{\a l} < \floor{(1-\b)l}, \a>\b$}.
\end{cases}$$
Note that $k<0$ in any case. One sees that $H^0(\P^1,\O^1(\log D)\ot \sO(k))$ has a basis
$$\frac{t^m}{1-t^l} \frac{dt}{t} \ot\o_n \quad (0 \le m \le l+k).$$
By \eqref{gr} and the above injectivity, we have
\begin{align*}
\dim \Gr_F^1H^1(\P^1,\sE^{(n)})
&=\dim H^0(\P^1,\O^1(\log D)\ot \sO(k))-\dim H^0(\P^1, \sO(i))\\
&=(l+k+1)-(i+1)=l+k-i . 
\end{align*}
By Corollary \ref{can ext fil} (i) and Lemma \ref{res gr^1}, we obtain the desired formula. 
\end{proof}

\begin{cor}
Assume that $p<l$ and $p>2$ when $a=b$. Then we have 
$$\Gr_F^1(H_\dR^2(X)/\angle{Z})^{(n)} \ne 0$$ 
for any $n=1,\dots, p-1$. 
\end{cor}

\begin{proof}
If $a\ne b$, then $\floor{|\a-\b|l} \ge \floor{\frac{l}{p}} \ge 1$. 
If $a=b$, then $\a \ne 1-\a$ since $p>2$, and hence 
$|\floor{\a l}-\floor{(1-\a)l}| \ge 1$. 
\end{proof}

\begin{ppn}\label{hodge number 0}
For each $n=1,\dots, p-1$, we have
$$
\dim \Gr_F^0(H_\dR^2(X)/\angle{Z})^{(n)}
=\min\{\floor{(1-\a)l}, \floor{\b l}\}-\max\{0, \floor{(\b-\a)l}\}.$$
\end{ppn}

\begin{proof}
By \eqref{gr}, Corollary \ref{res F^2} and Lemma \ref{res gr^1}, we have
$$\Gr_F^0(H_\dR^2(X)/\angle{Z})^{(n)}=H^1(\P^1,\Gr_F^0\sH_e^{(n)})=H^1(\P^1,\sO(k))$$
where $k$ is as in the proof of Proposition \ref{hodge number 1}.
Since $k<0$, we have
$$\dim H^1(\P^1,\sO(k)) = \dim H^0(\P^1,\sO(-k-2))=-k-1.$$
Hence the proposition follows. 
\end{proof}

\begin{rmk}
In fact, Proposition \ref{hodge number 0} is equivalent to the dimension formula in Proposition \ref{F^2 basis}. 
Note that the complex conjugation switches $n$ (resp. $\a$, $\b$) and $p-n$ (resp. $1-\a$, $1-\b$). 
\end{rmk}

\begin{thm}\label{thm 0}
The de Rham--Hodge structure $H=(H^2(X)/\angle{Z})\ot_R K$ has CM by $K$, i.e. $\dim_K H_B=1$. 
\end{thm}

\begin{proof}
Combining Propositions \ref{F^2 basis}, \ref{hodge number 1} and \ref{hodge number 0}, 
one verifies that
$$\dim (H_\dR^2(X)/\angle{Z})^{(n)}=l-1$$ 
for each $n=1,\dots, p-1$. 
It follows that
$$\dim_\Q H_B \le (l-1)(p-1)=[K:\Q].$$
It remains to show that $H \ne 0$, for which it suffices to show that $\tau$ is not identical on $H_\dR^2(X)/\angle{Z}$. If $p<l$, this follows from Proposition \ref{F^2 basis} and Corollary \ref{F^2 non-trivial}. The general case follows from Proposition \ref{F^1 basis} below. 
\end{proof}


\section{Periods}

We compute the periods of our $H$ and verify the Gross--Deligne conjecture, for which it will suffice to consider $F^1H_\dR$.  

\subsection{Basis of $F^1H_\dR$}

Recall that by \eqref{key sequence}, we can identify $F^1(H_\dR^2(X)/\angle Z)^{(n)}$ with the elements of 
$F^1H^1(\P^1,\sE^{(n)})$ having trivial residues.  
Furthermore, they are identified with rational $2$-forms by the following lemma. 
Put $T_1=\P^1 \setminus\{0,\infty\}$. 

\begin{lem}\label{iota}
For each $n=1,\dots, p-1$, there is a natural injection
\begin{equation*}
\i \colon F^1(H_\dR^2(X)/\angle Z)^{(n)} \hra \G(T_1, \O^1(\log D)\ot F^1\sH_e^{(n)}).
\end{equation*}
\end{lem}

\begin{proof}
By \eqref{fil} and \eqref{key sequence}, it suffices to show the existence of an injection 
$$H^1(\P^1, F^1\sE^{(n)}) \hookrightarrow \G(T_1, \O^1(\log D)\ot F^1\sH_e^{(n)})$$
where we put
$F^1\sE=[F^1\sH_e\to \O^1(\log D)\ot \sH_e]$. 
Consider the commutative diagram
$$\xymatrix{
& 0 \ar[d]\\
& \O^1(\log D) \ot F^1\sH_e^{(n)} \ar[d] \\
F^1\sH_e^{(n)} \ar[d]^{=} \ar[r]^{\nabla\phantom{AAAA}} &  \O^1(\log D) \ot \sH_e^{(n)} \ar[d] \\
F^1\sH_e^{(n)} \ar[r]^{\ol\nabla\phantom{AAAAA}} &  \O^1(\log D) \ot \Gr^0_F\sH_e^{(n)} \ar[d] \\
& 0
}$$
where the right vertical sequence is exact. 
By Proposition \ref{can ext}, $\ol\nabla$ is an isomorphism on $T_1$. 
Therefore, we have an isomorphism
$$\G(T_1, \O^1(\log D)\ot F^1\sH_e^{(n)}) \os\simeq{\longrightarrow} H^1(T_1, F^1\sE^{(n)})$$
It remains to show the injectivity of  $H^1(\P^1, F^1\sE^{(n)})\to H^1(T_1, F^1\sE^{(n)})$. 
This follows from the fact that $H^1(\P^1,F^1\sE) \to H^1(\P^1,\sE)$ is injective and 
$H^1(\P^1,\sE) \to H^1(T_1,\sE)$ is an isomorphism. 
\end{proof}

Under the identification via $\i$, we have the following. 
\begin{ppn}\label{F^1 basis}
For each $n=1,\dots, p-1$, a basis of $F^1(H_\dR^2(X)/\angle Z)^{(n)}$ is given by 
$\{\o_{m,n} \mid m \in I_n^1\}$ where
$$I_n^1:=\begin{cases}
\{-\floor{(\b-\a)l},\dots,-1\}\cup \bigl\{1,\dots, \max\{\floor{\a l}, \floor{(1-\b)l}\}\bigr\} & \text{if $\a<\b$},\\
\bigl\{1,\dots, \max\{\floor{\a l}, \floor{(1-\b)l}\}\bigr\} & \text{if $\a\ge \b$}. 
\end{cases}$$
Recall that $\a=\{\frac{na}{p}\}$, $\b=\{\frac{nb}{p}\}$.
\end{ppn}

\begin{proof}
It is a routine to verify that $|I_n^1|=\dim F^1(H_\dR^2(X)/\angle Z)^{(n)}$ using Propositions \ref{F^2 basis} and \ref{hodge number 1}. 
Therefore, it suffices to show that $\o_{m,n} \in F^1(H_\dR^2(X)/\angle Z)^{(n)}$ if $m \in I_n^1$. 
We construct \v{C}ech cocycles representing elements of $H^1(\P^1,F^1\sE^{(n)})$ with trivial residues which correspond to $\o_{m,n}$. 
Take a covering
$\P^1=U_0\cup U_\infty$ where $U_0:=\P^1\setminus\{\infty\}$, $U_\infty:=\P^1\setminus\{0\}$; note that $T_1=U_0\cap U_\infty$. A \v{C}ech cocycle in this case is a triple 
$$(\psi, \vphi_0, \vphi_\infty) \in \G(T_1, F^1\sH_e^{(n)}) \oplus \bigoplus_{t=0,\infty} \G(U_t,\O^1(\log D) \ot \sH_e^{(n)})$$
satisfying
$\nabla\psi=\vphi_0|_{T_1}-\vphi_\infty|_{T_1}.$
We construct such cocycles in four ways. 
By Proposition \ref{gauss-manin}, we have
\begin{align}
&l^{-1}\nabla(t^m\o_n)\notag\\
&=(\m-1+\b)\o_{m,n}+\frac{1-\a}{1-t^l} \y_{m,n}\label{cech1}\\
&=\left(\m-\a-\frac{1-\a-\b}{1-t^l}\right)\o_{m,n}+\frac{t^l}{1-t^l}\left((1-\a-\b)\o_{m,n}+\frac{1-\a}{t^l}\y_{m,n}\right)\label{cech2}\\
&=\left(\m+(1-\b)\frac{t^l}{1-t^l}\right) \o_{m,n} -\frac{1}{1-t^l}\left((1-\b)\o_{m,n}-(1-\a)\y_{m,n}\right)\label{cech3}\\
&=\left(\m-\a+\b+(1-\a)\frac{1-t^l}{1-t^l}\right)\o_{m,n}-\frac{1-\a}{1-t^l}(\o_{m,n}-\y_{m,n})\label{cech4}.
\end{align}
Put
$$j=\max\{0,\ceil{(\a-\b)l}\}, \quad k=\min\{\floor{\a l},\floor{(1-\b)l}\}.$$
(i) Suppose that $\floor{\a l}\ge \floor{(1-\b)l}$. 
Let $\psi=l^{-1}t^m\o_n$, 
$$\vphi_0=(\m-1+\b)\o_{m,n}, \ 
\vphi_\infty=-\frac{1-\a}{1-t^l}\y_{m,n}.$$
By \eqref{cech1} and Corollary \ref{local hodge fil}, these define a cocycle if $j \le m \le \floor{\a l}$. By Proposition \ref{N}, it has no residues unless $m=0$, hence defines an element of $F^1(H_\dR^2(X)/\angle Z)^{(n)}$ if 
$$j \le m \le \floor{\a l}, \ m \ne 0.$$
(ii) 
Suppose that $\floor{\a l}< \floor{(1-\b)l}$. Then, by \eqref{cech2} and Corollary \ref{local hodge fil}, $\psi=l^{-1}t^m\o_n$, 
\begin{align*}
\vphi_0 &=\left(\m-\a-\frac{1-\a-\b}{1-t^l}\right)\o_{m,n},\\ 
\vphi_\infty& =-\frac{t^l}{1-t^l}\left((1-\a-\b)\o_{m,n}+(1-\a)t^{-l}\y_{m,n}\right)
\end{align*}
define a cocycle if $j \le m \le \floor{(1-\b)l}$. 
To kill the residues, we use Lemma \ref{kill residue} below. 
Then, by letting 
$$\vphi_0=(\m-\a)\o_{m,n}, \ \vphi_\infty=(1-\a-\b)\o_{m,n}-\frac{1-\a}{1-t^l}\y_{m,n},$$
we obtain an element of $F^1(H_\dR^2(X)/\angle Z)^{(n)}$ for 
$$j \le m \le \floor{(1-\b)l}, \ m\ne 0.$$ 
(iii) Suppose that $\a\le\b$. Then, by \eqref{cech3} and Corollary \ref{local hodge fil}, 
$\psi=-l^{-1}t^m\o_n$, 
$$\vphi_0=\frac{1}{1-t^l}((1-\b)\o_{m,n}-(1-\a)\y_{m,n}), \ 
\vphi_\infty=\left(\m+(1-\b)\frac{t^l}{1-t^l}\right)\o_{m,n}$$
define a cocycle if 
$$-\floor{(\b-\a)l} \le m \le k.$$ 
If $m<0$, we can kill the residues using Lemma \ref{kill residue}, and 
$$\vphi_0=(1-\b)\o_{m,n}-\frac{1-\a}{1-t^l}\y_{m,n}, \ \vphi_\infty =\m \o_{m,n}$$
define an element of $F^1(H_\dR^2(X)/\angle Z)^{(n)}$ for
$$-\floor{(\b-\a)l} \le m <0.$$ 
(iv) Finally, suppose that $\a>\b$. Then, by \eqref{cech4} and Corollary \ref{local hodge fil},  $-l^{-1}t^m\o_n$, 
$$\vphi_0=\frac{1-\a}{1-t^l}(\o_{m,n}-\y_{m,n}), \ \vphi_\infty=\left(\m-\a+\b+(1-\a)\frac{t^l}{1-t^l}\right)\o_{m,n}$$
define a cocycle if $0 \le m \le k$. If $m \ne 0$, we can use Lemma \ref{kill residue} to kill the residues and 
$$\vphi_0=(1-\a)\o_{m,n}-\frac{1-\a}{1-t^l}\y_{m,n}, \ \vphi_\infty=(\m-1+\b)\o_{m,n}$$
define an element of $F^1(H_\dR^2(X)/\angle Z)^{(n)}$ for 
$$0<m \le k.$$
Combining (iii) and (i) (or (ii)), we obtain the first case of the proposition. 
For the second case, combine (iv) and (i) (or (ii)), just noting that $k \ge j-1=\floor{(\a-\b)l}$. 
\end{proof}

\begin{lem}\label{kill residue}
If $j \le m<l$, $m \ne 0$, then
$$\frac{1}{1-t^l} \ot \o_{m,n} \in \G(\P^1, \O^1(\log D) \ot \sH_e^{(n)}),$$
and it has trivial residues at $t=0, \infty$. 
\end{lem}

\begin{proof}
This is immediate from Corollary \ref{local hodge fil} and Lemma \ref{res 0}. 
\end{proof}

\subsection{Period formula}

We prove the period formula which verifies the conjecture of Gross--Deligne \cite[Sect. 4]{gross} (but see Remark \ref{epsilon} below). 
We identify an embedding $\chi \colon K \hookrightarrow \C$ with the element $h \in(\Z/lp\Z)^\times$ such that $\chi(\z_{lp})=\z_{lp}^h$, and write $H^{(h)}$ instead of $H^\chi$. 
For each $h \in (\Z/lp\Z)^\times$, let $(p(h),2-p(h))$ be the Hodge type of $H^{(h)}$. 
Put $K'=\Q(\mu_{2lp})$ ($K=K'$ if $lp$ is odd). 
\begin{thm}\label{thm 1}
Define a function $\ve\colon \Z/lp\Z \ra \Z$ by
\begin{equation*}
\ve(i) = \begin{cases} 1 & \text{if $i\equiv lb, p, l(p-b), l(b-a)+p \pmod{lp}$},\\
-1 & \text{if $i\equiv lb+p, l(p-a)+p \pmod{lp}$}, \\
0 & \text{otherwise}.
\end{cases}
\end{equation*}
Then, for any $h \in (\Z/lp\Z)^\times$, we have
$$p(h)=\sum_{i\in \Z/lp\Z} \ve(i) \left\{-\frac{hi}{lp}\right\}$$
and
$$\Per(H^{(h)}) \sim_{K'^\times} \prod_{i\in \Z/lp\Z} \G\left(\left\{\frac{hi}{lp}\right\}\right)^{\ve(i)}.$$
\end{thm}

\begin{proof}
For real numbers $x$, $y$ with $0<x,y<1$, $x+y\neq 1$, put
$$\d(x,y):=\{-x\}+\{-y\}-\{-(x+y)\}=\begin{cases}
1 & \text{if $x+y<1$}, \\ 0 & \text{if $x+y>1$}.
\end{cases}
$$
Then we have
$$\vphi(h):= \sum_{i} \ve(i) \left\{-\frac{hi}{lp}\right\}=\d\left(\b,\mu\right)+\d\left(1-\b,\{\b-\a+\mu\}\right)$$
where we put $\a=\{ha/p\}$, $\b=\{hb/p\}$, $\mu=\{h/l\}$.  
Firstly, we have $\vphi(h)=2$ if and only if
$$\b+\mu<1, \ 1-\b+\{\b-\a+\mu\}<1.$$
Letting $m=l\mu$, the first condition becomes $m <(1-\b)l$, i.e. $m \le \floor{(1-\b)l}$. 
Similarly, the second condition is equivalent to
$$(\a\le \b, \ m<\a l) \quad \text{or} \quad (\a>\b, \ (\a-\b)l <m<\a l).$$
Comparing with Proposition \ref{F^2 basis}, we have $p(h)=2$ if and only if $\vphi(h)=2$. 
Secondly, since 
$$p(h)+p(-h)=\vphi(h)+\vphi(-h)=2,$$ 
we have $p(h)=0$ if and only of $\vphi(h)=0$. 
Since $p(h), \vphi(h) \in \{0,1,2\}$, we have $p(h)=\vphi(h)$ for any $h$. 

For the second statement, we compute the periods over the $2$-cycle 
$$(1-\tau)_*(1-\s)_*\D_1.$$ 
Since $(1-\z_l)(1-\z_p)$ is invertible in $K$, it reduces to the periods over $\D_1$ (Proposition \ref{2-period} (i)). 
Consider firstly the two cases:
\begin{enumerate}
\item $\a\le\b$ and $p(h) \ge 1$, 
\item $\a>\b$ and $p(h)=2$. 
\end{enumerate}
In both cases, by Propositions \ref{F^2 basis} and \ref{F^1 basis}, $H^{(h)}$ is generated by $\o_{m,n}$ satisfying $\ceil{(\a-\b)l} \le m$, which is equivalent to $\a-\b < \mu:=m/l$. 
This is the assumption of Proposition \ref{2-period} (i) and we obtain the desired formula. 

The other cases are reduced to the above ones. If we replace $\chi$ with $\chi^{-1}$, then $h$ (resp. $\a$, $\b$, $p(h)$) is replaced with $-h$ (resp. $1-\a$, $1-\b$, $2-p(h)$). 
By Lemma \ref{autodual} below, the cup-product $H^2(X) \ot H^2(X) \to \Q(-2)$ induces an auto-duality on $H$,  
under which $H^\chi$ is dual to $H^{\chi^{-1}}$. 
Hence we have
$$\Per(H^{(h)})\cdot \Per(H^{(-h)}) \sim_{K^\times} (2 \pi i)^2.$$
On the other hand, recall
$$\G(x)\G(1-x) = \frac{\pi}{\sin \pi x} \sim_{K'^\times} 2\pi i$$
for any $x \in \frac{1}{lp}\Z \setminus \Z$. 
Therefore, the case where $\a \le \b$ and $p(h)=0$ (resp. $\a>\b$ and $p(h) \ge 1$) is equivalent to the case (ii) (resp. (i)). 
\end{proof}

\begin{lem}\label{autodual}
Put $H^2(X)_Z = \Ker(H^2(X) \to H^2(Z))$. Then, the composition
$$H^2(X)_Z \hookrightarrow H^2(X) \twoheadrightarrow H^2(X)/\angle{Z}$$
induces an isomorphism of de Rham--Hodge structures
$H^2(X)_Z \ot_RK \simeq H$. 
\end{lem}

\begin{proof}
This follows from the fact that the kernel of the composite
$$H_Z^2(X,\C) \to H^2(X,\C) \to H^2(Z,\C)$$
is one-dimensional by Zariski's lemma (cf. \cite[III, (8.2)]{barth}). 
\end{proof}

\begin{rmk}\label{epsilon}
Our definition of $\e$ is slightly different from \cite{gross}; $\e(i)$ here is $\e(-i)$ of loc. cit., 
where Gross looks at the values $\G\left(1-\{{hi}/{lp}\}\right)^{\ve(i)}$.  
The former conforms to the definition of the Stickelberger element as 
$$\sum_{h\in(\Z/N\Z)^\times} \left\{-\frac{h}{N}\right\} \s_h^{-1}$$
where $\s_h \in \mathrm{Gal}(\Q(\mu_N)/\Q)$ sends an $N$th root of unity to its $h$th power.    
\end{rmk}

\section{Regulators}

After explaining the regulator map we consider, we prove Theorem \ref{thm intro 2} from the introduction and its consequences on the non-vanishing. 

\subsection{Formulation}

The Deligne cohomology of $X_\C:=X\times_{\Spec \ol\Q} \Spec \C$ with coefficients in $\Q(2)$ is defined to be the hypercohomology of the complex 
$$\Q(2) \to \sO_{X_\C} \to \O^1_{X_\C/\C}$$
where $\Q(2):=(2\pi i)^2\Q$ is placed in degree $0$. 
Consider the Beilinson regulator map \cite{beilinson} from the motivic cohomology
$$r_\sD\colon H^3_{\sM}(X,\Q(2)) \to  H^3_{\sD}(X_\C,\Q(2)). $$
We have a natural isomorphism
$$H^3_\sD(X_\C,\Q(2))\simeq H^2(X,\C)/(F^2+H^2(X,\Q(2))),$$
and the Carlson isomorphism
$$H^2(X,\C)/(F^2+H^2(X,\Q(2))) \simeq \Ext^1_\MHS(\Q,H^2(X, \Q(2))).$$
Here, $\MHS$ denotes the abelian category of $\Q$-mixed Hodge structures. 
By Poincar\'e duality $H^2(X,\Q(2))\simeq H_2(X,\Q)$, we obtain an identification 
$$H^3_\sD(X_\C,\Q(2)) \simeq \Ext^1_\MHS(\Q,H_2(X,\Q)).$$

Let $Z \subset X$ be as before and consider the regulator map 
$$r_{\sD,Z} \colon H^3_{\sM,Z}(X,\Q(2)) \to H^3_{\sD, Z}(X,\Q(2)) \simeq H_1(Z,\Q)$$
from the motivic cohomology supported on $Z$ (see \cite{asakura}, Sect. 2).  
Since $H_1(X,\Q)=0$ by Proposition \ref{H^1=0}, we have an exact sequence of mixed Hodge structures
$$H_2(Z,\Q) \to H_2(X,\Q) \to H_2(X,Z;\Q) \os\pd\to H_1(Z,\Q) \to 0.$$
If we denote the image of the first map by $\angle{Z}$, we have the connecting homomorphism
$$\r\colon H_1(Z,\Q) \cap H^{0,0} \to \Ext^1_\MHS(\Q,H_2(X,\Q)/\angle{Z}).$$
By the lemma and the remark below, $\rho$ describes the restriction of $r_\sD$ to the image of $H^3_{\sM,Z}(X,\Q(2))$. 
\begin{lem}
The diagram below is commutative up to sign:
\begin{equation*}
\xymatrix{
H^3_{\sM,Z}(X,\Q(2)) \ar[r]^{r_{\sD,Z}} \ar[d] & H_1(Z,\Q) \cap H^{0,0} \ar[r]^{\r\phantom{AAAAAA}} & \Ext^1_\MHS(\Q,H_2(X,\Q)/\angle{Z}) \\ 
H^3_{\sM}(X,\Q(2))
 \ar[r]^{r_\sD} & H^3_{\sD}(X_\C,\Q(2)) \ar[r]^{\simeq\phantom{AAA}} & \Ext^1_{\MHS}(\Q,H_2(X,\Q)) \ar[u]
}
\end{equation*}
where the vertical maps are the natural ones. 
\end{lem}
\begin{proof}
See \cite{asakura-sato}, Theorem 11.2. 
\end{proof}

\begin{rmk}
The right vertical arrow is surjective since $\Ext_\MHS^2=0$. 
Its kernel is topologically generated by decomposable elements, i.e. the image of 
$$(\mathrm{CH}_1(Z) \ot \ol\Q^\times)\ot_\Z\Q \to H^3_{\sM,Z}(X,\Q(2)).$$ 
Also, it is not difficult to show that $r_{\sD, Z}$ is surjective (see \cite{asakura}). 
\end{rmk}

\subsection{Regulator formula}

Now, we regard the extension classes as functionals (up to period functionals). Let
$H^2(X)_Z=\Ker(H^2(X) \to H^2(Z))$ as before.
Since 
$H^2(X,\Q)_Z \simeq (H_2(X,\Q)/\angle{Z})^*$, 
we have
$$\Ext^1_\MHS(\Q,H_2(X,\Q)/\angle{Z}) \simeq (F^1H^2(X,\C)_Z)^*/ \Image H_2(X,\Q)$$
where $*$ denotes the $\C$-linear dual. 
By Lemma \ref{autodual}, $\r$ induces a map
$$\r\colon (H_1(Z,\Q)\cap H^{0,0}) \ot_RK \to (F^1H_\C)^*/H_B^\vee$$
where $H_\C:=H_B \ot_\Q\C$ and $H^\vee$ denotes the dual de Rham--Hodge structure of $H$. 

Put $Z_1= \bigsqcup_{\z \in \mu_l} X_\z$. 
We shall describe the restriction of $\r$ to $H_1(Z_1,\Q) \ot_RK$. 
Recall that $H_1(Z_1,\Q) \subset H^{0,0}$ (Corollary \ref{totdeg}). 
We have in fact the following. 

\begin{lem}We have an isomorphism 
$$H_1(Z_1,\Q) \ot_R K \os\simeq\to H_1(Z,\Q) \ot_R K.$$
\end{lem}

\begin{proof}
By Proposition \ref{N}, $\tau$ acts trivially on $H_1(X_0,\Q)$ and $H_1(X_\infty,\Q)=0$. 
\end{proof}

Let $(1-\s)_*\D_0 \in H_2(X,Z_1;\Q)$ be the Lefschetz thimble defined in Sect. 2.5, and 
let $H_2(X,Z_1;\Q)_{\mathrm{Lef}} \subset H_2(X,Z_1;\Q)$ denote the $R$-submodule generated by this element. 

\begin{lem}
The restriction of the boundary map
$$\pd \colon H_2(X,Z_1;\Q)_{\mathrm{Lef}} \ot_R K \to H_1(Z_1,\Q) \ot_R K$$
is surjective and $H_1(Z_1,\Q) \ot_R K$ is one-dimensional over $K$. 
\end{lem}

\begin{proof}
By Proposition \ref{N}, $\dim_\Q H_1(X_\z,\Q)=p-1$ for $\z \in \mu_l$. 
Since $\tau$ permutes the components of $Z_1$, 
$H_1(Z_1,\Q) \ot_R K$ is one-dimensional over $K$. 
Whereas $\k_0$ and $\k_1$ generate $H_1(X_t,\Q)$ (Proposition \ref{kappa} (ii)),
$\k_1$ vanishes as $t \to 1$ by definition. 
Therefore $\k_0$ does not vanish, i.e. $\pd((1-\s)_*\D_0)$ is non-trivial in $H_1(X_1,\Q)$, 
hence so is in $H_1(Z_1,\Q) \ot_R K$. 
\end{proof}

%

Now we state our main theorem. For $x \in K$, let $x_*$ (resp. $x^*$) denote its action on homology (resp. cohomology). 
Since $1-\z_p$ is invertible in $K$, we write 
$$((1-\z_p)^{-1})_* (1-\s)_*\D_0 \in H_1(X,Z_1;\Q)\ot_RK$$ simply as $\D_0$. 
For each $m$ and $n$, define an embedding $\chi_{m,n}\colon K\hookrightarrow \C$ by 
$$\chi_{m,n}(\z_l)=\z_l^m, \quad \chi_{m,n}(\z_p)=\z_p^n.$$

\begin{thm}\label{thm 2}
Let $\g \in H_1(Z_1,\Q) \ot_RK$ and take $x \in K$ such that $\g=x_*\pd\D_0$.   
Let $\{\o_{m,n} \mid n=1,\dots, p-1, m \in I_n^1\}$ be the basis of $F^1H_\dR$ 
given in Proposition \ref{F^1 basis}. 
Then we have
$$\r(\g)(\o_{m,n}) = \chi_{m,n}(x)
\frac{B(1-\a,\b)}{l(\b-\a+\mu)}\cdot \F{1-\a,\b,\b-\a+\mu}{1-\a+\b,\b-\a+\mu+1}{1}$$
where $\a=\{\frac{na}{p}\}$, $\b=\{\frac{nb}{p}\}$, $\m=\frac{m}{l}$. 
\end{thm}

\begin{proof}
We apply Theorem \ref{thm apdx} of the appendix (see also \cite[Theorem 4.1]{asakura}) to our situation where $D=Z_1$ and  $X^\circ = X \setminus(X_0 \cup X_\infty)$ (see the proof of Lemma \ref{iota}). 
Note that $H_\C \simeq H_\dR^2(X_\C)_0 \ot_R K$ by Lemma \ref{autodual} since $\tau$ acts trivially on $H_\dR^2(e(\P^1_\C))$ (see Sect. 7.2 for the notations). 

Put $\G=(1-\tau)_*(1-\s)_*\D_0$. 
Since $\G \in H_2(X,Z_1;\Q)$ does not necessarily comes from $H_2(X^\circ,Z_1;\Q)$, we take a detour. 
Let $\G'$ be the Lefschetz thimble given by sweeping $(1-\s)_*\d_0$ along the path $\k_1+\k_2+\k_3$ in $T\setminus\{0,\infty\}$ where $\k_1$ is the line segment from $\z$ to $\e\z$ ($\e>0$), $\k_2$ is the arc from $\e\z$ to $\e$ and $\k_3$ is the line segment from $\e$ to $1$. 
Then $\G' \in H_2(X^\circ,Z_1;\Q)$ and $\g:=\pd(\G)=\pd(\G')$. 
Theorem \ref{thm apdx} yields
$$\r(\g)(\o_{m,n})=\int_{\G'} \o_{m,n}. $$
The right integral is computed similarly as Proposition \ref{2-period} (ii), and letting $\e \to 0$ we obtain the theorem for $x=(1-\z_l)(1-\z_p)$. 
The general case follows by the cyclicity of $H_1(Z_1,\Q) \ot_RK$. 
\end{proof}

\subsection{Non-vanishing}

We prove the non-vanishing of $\r$ under a mild assumption. 
The situation is different depending whether $a+b=p$ or not. 

If $a+b\ne p$, the regulator does not vanish even in the Deligne cohomology with $\R$-coefficients, 
or equivalently, the extension group of $\R$-mixed Hodge structures
$$\Ext_{\R\MHS}^1(\R,H_\R) \simeq (F^1H_\C)^*/ H^\vee_\R$$
where $H_\R=H_B \ot_\Q \R$, $H_\C=H_B \ot_\Q \C$. Note that 
$$\dim_\R (F^1H_\C)^*/ H^\vee_\R=\dim_{\ol\Q} \Gr_F^1 H_\dR.$$
Let 
$$\r_\R \colon H_1(Z_1,\Q)\ot_R K \to (F^1H_\C)^*/ H^\vee_\R$$
be the composition of $\r$ and the natural surjection. 

\begin{thm}\label{thm 3} 
Suppose that $p<l$ and $a+b \ne p$ (so $p>2$). Then $\r_\R$ is non-trivial. 
In particular, 
$$\dim_\Q \r_\R(H_1(Z_1,\Q)\ot_R K)=(l-1)(p-1).$$ 
\end{thm}

\begin{proof}
By restricting the functionals to $F^1H_\R:=F^1H_\C \cap H_\R$ and taking the imaginary part, we obtain a $K\cap \R$-linear map
$$\r_\R'\colon H_1(Z_1,\Q)\ot_R K \to \Hom(F^1H_\R, i\R).$$
For each $n=1,\dots, p-1$, we have $\a\ne 1-\b$ by the assumption. 
Hence $|\a-(1-\b)| \ge 1/p >1/l$ and there exists an $m$ satisfying 
\begin{align}\label{good m}
\min\{\floor{\a l}, \floor{(1-\b)l}\} < m \le \max\{\floor{\a l}, \floor{(1-\b)l}\}.
\end{align}
Then we have $\o_{m,n} \in \Gr_F^1 H_\dR$
by Propositions \ref{F^2 basis} and \ref{F^1 basis}. 
Since $m>\floor{(\a-\b)l}$, we have $\m:=m/l>\a-\b$, hence we can apply Proposition \ref{2-period} (i) to compute the period: 
$$\O_{m,n}:=\int_{\D_1} \o_{m,n} = -\frac{(-1)^{p\b}}{l} B(\b,\m)B(1-\b,\b-\a+\m).$$
Put a normalization
$$\wt\o_{m,n} = \O_{m,n}^{-1} \o_{m,n}.$$
Then we have 
$$\int_{x_*\D_1} \wt\o_{m,n} = \int_{\D_1} x^* \wt\o_{m,n} = \chi_{m,n}(x)$$
for any $x \in K$. 
If we let $n'=p-n$, $\a'=\{n'a/p\}=1-\a$, $\b'=\{n'b/p\}=1-\b$, 
$m'=l-m$ and $\m'=\{m'/l\}=1-\m$, then these satisfy the assumption \eqref{good m}. Hence 
$\wt\o_{m',n'}$ is defined and we have 
$$\int_{x_* \D_1} \wt\o_{m',n'} = \ol{\chi_{m,n}(x)}$$
for any $x \in K$.  
Since $H_B^\vee$ is generated as a $K$-module by
$$((1-\z_l)^{-1}(1-\z_p)^{-1})_*(1-\tau)_*(1-\s)_*\D_1,$$
which we simply denote $\D_1$ as before, we have $\ol{\wt\o_{m,n}}=\wt\o_{m',n'}$ and hence
$$\wt\o_{m,n}+\wt\o_{m',n'} \in F^1H_\R.$$
Put the regulator as
$$R_{m,n}:= \int_{\D_0} \o_{m,n} = \frac{B(1-\a,\b)}{l(\b-\a+\mu)}\cdot \F{1-\a,\b,\b-\a+\mu}{1-\a+\b,\b-\a+\mu+1}{1}.$$
By Theorem \ref{thm 2}, for any $\g \in H_1(Z_1,\Q)$ corresponding to $x \in K$ as in loc. cit., we have
\begin{align*}
\r_\R'(\g)(\wt\o_{m,n}+\wt\o_{m',n'})
&=\Im\left(\chi_{m,n}(x) \O_{m,n}^{-1} R_{m,n} +\ol{\chi_{m,n}(x)} \O_{m',n'}^{-1} R_{m',n'}\right)
\\&=\Im(\chi_{m,n}(x)) \left(\O_{m,n}^{-1} R_{m,n}-\O_{m',n'}^{-1} R_{m',n'}\right). 
\end{align*}
Since $\O_{m,n}\O_{m',n'}<0$ and $R_{m,n}, R_{m',n'}>0$, 
the above does not vanish for $x \in K \setminus \R$. 
Hence $\r_\R$ is non-trivial. Since $\r_\R$ is $K$-linear, the second assertion follows. 
\end{proof}

The non-vanishing of $\r$ is a more subtle problem. 
For the case $a+b=p$, we have the following criterion. 

\begin{ppn}\label{criterion}
Let $p$, $l$ be distinct prime numbers and suppose that $a+b=p$. 
If $\r$ is trivial, then there exists an $x \in K$ such that 
$$R_{m,n} = \chi_{m,n}(x) \O_{m,n}$$
for any $n=1,\dots, p-1$ and $m \in I_n^1$ such that $\frac{m}{l}>\{\frac{na}{p}\}-\{\frac{nb}{p}\}$. 
\end{ppn}

\begin{proof}
Let $\g=\pd\D_0$ and suppose that $\r(\g)=0$. 
Since $H_B^\vee$ is generated by $\D_1$ over $K$, 
there exists an $x\in K$ such that $\r(\g)$ is represented by the functional $\int_{x_*\D_1}$. 
If $m$, $n$ are as in the statement, then 
$$\int_{x_*\D_1}\o_{m,n} = \int_{\D_1} x^*\o_{m,n} = \chi_{m,n}(x) \O_{m,n}$$
by the definition. Hence the proposition follows. 
\end{proof}

\begin{ex}If $p=2$, then $\a=\b=1/2$ and $Y$ is nothing but the Legendre family of elliptic curves. By Proposition \ref{hodge number 1}, we have $\Gr_F^1H_\dR=0$ and the Deligne cohomology with $\R$-coefficients is trivial. 
Since  the condition $\frac{m}{l}>\{\frac{na}{p}\}-\{\frac{nb}{p}\}$ ($=0$) is automatically satisfied, 
Proposition \ref{criterion} is in fact an equivalence. 
If $l=3$ for example, then $\rho$ is trivial if and only if
$$\sqrt{3} \left(\frac{\G(\frac{5}{6})}{\G(\frac{1}{3})}\right)^2 \cdot
\F{\frac{1}{2}, \frac{1}{2},\frac{1}{3}}{1,\frac{4}{3}}{1}
 \in \Q.$$
Here we used $\Q(\z_3) \cap i\R = \sqrt{3}i \Q$. 
\end{ex}

\section{Appendix: Fibration of Curves and Extension of Motives}

In this appendix, we give a short exposition of a technique developed in 
\cite{asakura} which is used in the proof of the regulator formula (Theorem \ref{thm 2}).

\subsection{Relative cohomology}

Let $V$ be a quasi-projective smooth surface over $\C$.
Let $D\subset V$ be a chain of curves.
Let $\pi \colon\wt{D}\to D$ be the normalization and $\Sigma\subset D$ be
the set of singular points. Let $s\colon \wt{\Sigma}:=\pi^{-1}(\Sigma)
\hra\wt{D}$ be the inclusion.
There is an exact sequence
$$
0\to \cO_{D}\os{\pi^*}{\to} \cO_{\wt{D}}\os{s^*}{\to} 
\C_{\wt{\Sigma}}/\C_\Sigma\to 0
$$
where $\C_{\wt{\Sigma}}=\mathrm{Maps}(\wt{\Sigma},\C)=\Hom(\Z\wt{\Sigma},\C)$ and 
$\pi^*$, $s^*$ are the pull-backs. 
For a smooth manifold $M$, let 
$\cA^q(M)$ denote the space of smooth differential $q$-forms on $M$
with coefficients in $\C$. 
We define $\cA^\bullet(D)$ to be the mapping fiber of 
$s^*\colon \cA^\bullet(\wt{D})\to \C_{\wt{\Sigma}}/\C_\Sigma$:
$$
\cA^0(\wt{D})
\os{s^*\op d}{\lra} \C_{\wt{\Sigma}}/\C_\Sigma\op \cA^1(\wt{D})
\os{0\op d}{\lra} \cA^2(\wt{D})
$$
where the first term is placed in degree 0.
Then
$$
H^q_\dR(D)=H^q(\cA^\bullet(D))
$$
is the de Rham cohomology of $D$, which fits into the exact sequence
$$
\cdots\to H^0_\dR(\wt{D})\to \C_{\wt{\Sigma}}/\C_\Sigma
\to H^1_\dR(D)\to H^1_\dR(\wt{D})\to\cdots.
$$
We have the natural pairing 
\begin{equation*}\label{pairingD}
\langle \ , \ \rangle_D\colon H_1(D,\Z)\otimes H^1_\dR(D)\to \C,
\quad
\gamma\ot z\mapsto \int_\gamma\eta-c(\partial(\pi^{-1}\gamma))
\end{equation*}
where $z$ is represented by $(c,\eta)\in \C_{\wt{\Sigma}}/\C_\Sigma\op\cA^1(\wt{D})$ with $d\eta=0$
and $\partial$ denotes the boundary of homology cycles.

We define $\cA^\bullet(V,D)$
to be the mapping fiber of 
$\wt i^*\colon \cA^\bullet(V)\to \cA^\bullet(\wt D)$, the pull-back by $\wt i \colon \wt D\to V$:
$$
\cA^0(V)\os{\cD_0}{\to} \cA^0(\wt{D})\op \cA^1(V)
\os{\cD_1}{\to} \C_{\wt{\Sigma}}/\C_\Sigma\op\cA^1(\wt{D})\op \cA^2(V)\os{\cD_2}{\to}\cdots.
$$
Then the relative de Rham cohomology is defined by
\begin{equation*}\label{exp-9-0}
H^q_\dR(V,D)=H^q(\cA^\bullet(V,D))
\end{equation*}
and fits into the exact sequence
\begin{equation}\label{exp-9}
\cdots\to H^{q-1}_\dR(D)\to H^q_\dR(V,D)\to H^q_\dR(V)\to H^q_\dR(D)\to\cdots. 
\end{equation}
An element of $H^2_\dR(V,D)$ is represented by 
\begin{equation}\label{element}
(c,\eta,\omega)
\in \C_{\wt{\Sigma}}/\C_\Sigma\op\cA^1(\wt{D})\op \cA^2(V)
\end{equation}
which satisfies $\wt i^*\omega=d\eta$ and $d\omega=0$.
The natural pairing
\begin{equation*}\label{pairingVD1}
\langle \ , \ \rangle_{V,D} \colon H_2(V,D;\Z)\ot H_\dR^2(V,D)\to \C
\end{equation*}
is given by
\begin{equation*}\label{pairingVD2}
\angle{\G,z}_{V,D}
=\int_\Gamma \omega- \angle{\pd\G,(c,\eta)}_D 
=\int_\Gamma \omega-\int_{\partial\Gamma}\eta+c(\pd(\pi^{-1}(\partial\Gamma))).
\end{equation*}

The complexes $\cA^\bullet(V)$ and $\cA^\bullet(D)$ are canonically equipped with
Hodge and weight filtrations, and then 
$(\Q_V,\cA^\bullet(V),F^\bullet,W_\bullet)$ 
and $(\Q_D,\cA^\bullet(D),F^\bullet,W_\bullet)$ become cohomological mixed
Hodge complexes in the sense of \cite[(8.1.2)]{hodge III}.
The Hodge and weight filtrations on $\cA^\bullet(V,D)$ are induced from them
and the data
$(\Q_{V,D},\cA^\bullet(V,D),F^\bullet,W_\bullet)$ becomes a cohomological mixed
Hodge complex as well.
Hence we have an exact sequence
\begin{equation*}
\cdots\to H^{q-1}(D,\Q)\to H^q(V,D;\Q)\to H^q(V,\Q)\to H^q(D,\Q)\to\cdots
\end{equation*}
of mixed Hodge structures which is compatible with \eqref{exp-9}. 
Taking its dual, we obtain an exact sequence
\begin{equation*}
0\to H_2(V,\Q)/H_2(D)\to H_2(V,D;\Q)\os{\partial}{\to} H_1(D,\Q)\to H_1(V,\Q).
\end{equation*}
Since $H_1(V,\Q) \cap H^{0,0}=0$, 
we obtain the coboundary map
\begin{equation*}
\rho_{V,D} \colon H_1(D,\Q) \cap H^{0,0} \to \Ext^1_\MHS(\Q,H_2(V,\Q)/H_2(D))
\end{equation*}
to the extension group of mixed Hodge structures.
If we put
$$H^2_\dR(V)_D:=\Ker[H^2_\dR(V)\to H^2_\dR(D)],$$ 
then we have the Carlson isomorphism
\begin{equation*}
\Ext^1_\MHS(\Q,H_2(V,\Q)/H_2(D))\simeq\Coker\left[H_2(V,\Q)\to (F^1H^2_\dR(V)_D)^*\right]
\end{equation*}
where $*$ denotes the $\C$-linear dual and the map is the natural pairing. 
Under this identification, the map $\rho_{V,D}$ is described as follows.
For $\gamma\in H_1(D,\Q)\cap H^{0,0}$, take a $\Gamma\in H_2(V,D;\Q)$ such that $\partial(\Gamma)=\gamma$.
Then we have
\begin{equation}\label{exp-14}
\rho_{V,D}(\gamma)=\left[\omega\mapsto\angle{\Gamma,\omega_{V,D}}_{V,D}
\right]
\end{equation}
where $\omega_{V,D}\in F^1H^2_\dR(V,D)$ is a lifting of $\omega$,  
on which the pairing does not depend.  

\subsection{Rational forms}
For a given $\o$, it is usually complicated to compute an analytic lifting $\omega_{V,D}$ explicitly. 
In the following situation, we shall be able to associate a {\em rational} $2$-form via Deligne's canonical extension, which gives a simple expression of $\r_{V,D}$.  

Let $C$ be a projective smooth curve over $\C$ and $f\colon X\to C$ be a fibration of curves with connected general fiber which admits a section $e\colon C\to X$. 
From now on, we use the algebraic de Rham cohomology groups (see \cite{hartshorne}) and identify them with the analytic ones in the previous paragraph. 
For a Zariski open set $S\subset C$, let $V=f^{-1}(S)$ and put
\begin{align*}
& H^2_\dR(V)_0=\Ker\left[H^2_\dR(V)\to \prod_{s\in S}H^2_\dR(f^{-1}(s))\times H^2_\dR(e(S))\right],
\\& H^2_\dR(V,D)_0=\Ker\left[H^2_\dR(V,D)\to H^2_\dR(V)/H^2_\dR(V)_0\right].
\end{align*}
Then we have an exact sequence of mixed Hodge structures
\begin{equation}\label{pairingVD4-2}
H^1_\dR(V)\to H^1_\dR(D)\to H^2_\dR(V,D)_0\to H^2_\dR(V)_0\to0. 
\end{equation}
The arrows are strictly compatible with the Hodge and weight filtrations. 
In particular, $F^1H_\dR^2(V,D)_0 \to F^1H_\dR^2(V)_0$ is surjective. 
Later, we shall use the following. 

\begin{lem}\label{lem-birat}
Let $g\colon V'\to V$ be a birational transformation which is isomorphic outside $D$ and put $D'=g^{-1}(D)$.
Then the pull-back $g^*$ induces
isomorphisms $H^2_\dR(V)_0\simeq H_\dR^2(V')_0$
and $H^2_\dR(V,D)_0\simeq H_\dR^2(V',D')_0$.
\end{lem}
\begin{proof}
By \eqref{pairingVD4-2} it is enough to show isomorphisms
$$
H^1_\dR(V)\simeq H^1_\dR(V'),\quad
H^1_\dR(D)\simeq H^1_\dR(D'),\quad
H^2_\dR(V)_0\simeq H^2_\dR(V')_0.
$$
The first one is an easy exercise.
Let $X'$ be a smooth compactification of $V'$ such that $X'\setminus D' \simeq X\setminus D$ and
consider the commutative diagram with exact rows 
$$\xymatrix{
H^2_\dR(X')\ar[d]_{g_*}\ar[r]^{a^2\quad} &H^2_\dR(X'\setminus D')\ar[r] \ar@{=}[d] 
&H^\dR_1(D')\ar[d]_{g_*}\ar[r]&
H^3_\dR(X') \ar[d]_{g_*}^\simeq\ar[r]^{a^3\quad}& H^3_\dR(X'\setminus D')\ar@{=}[d]
\\
H^2_\dR(X)\ar[r]^{b^2\quad} &
H^2_\dR(X\setminus D)\ar[r]&H^\dR_1(D)\ar[r]
&H^3_\dR(X)\ar[r]^{b^3 \quad} &H^3_\dR(X\setminus D). 
}$$
The second isomorphism follows from the fact that 
$\Image(a^n)=\Image(b^n)=W_nH^n_\dR(X\setminus D)$. 
The last isomorphism follows from the commutative diagram
$$\xymatrix{
0\ar[r] &H^2_\dR(V)_0\ar[r] \ar[d]_{g^*}&
H^2_\dR(V\setminus D)_0\ar[r]\ar@{=}[d] &H^\dR_1(D)\ar[d]_{g^*}^\simeq\\
0\ar[r] &H^2_\dR(V')_0\ar[r] &H^2_\dR(V'\setminus D')_0\ar[r] &H^\dR_1(D')
}$$
with exact rows.
\end{proof}

Now, fix a Zariski open set $S\subset C$ such that $U:=f^{-1}(S)\to S$ is smooth.
Put $T=C\setminus S$ and $Z=X\setminus U$. Let  
$$\nabla\colon \cH_e\to\Omega^1_C(\log T)\ot\cH_e$$
be the Deligne canonical extension
of the Gauss--Manin connection $(\cH=R^1f_*\Omega^\bullet_{U/S},\nabla)$. 
Put $F^1\cH_e=j_*F^1\cH\cap \cH_e$ where $j\colon S \hra C$ and $\Gr_F^0\cH_e=\cH_e/F^1\cH_e$.
Let
\begin{equation*}
\ol{\nabla}\colon F^1\cH_e\to \Omega^1_C(\log T)\ot \Gr_F^0\cH_e
\end{equation*}
be the $\sO_C$-linear map induced from $\nabla$. 
In what follows, we assume the following:
\begin{itemize}
\item[(*)]
The map $\ol\nabla$ is generically bijective.
\end{itemize}
Let $C^\circ\subset C$ be a Zariski open set on which 
$\ol\nabla$ is bijective and put $X^\circ:=f^{-1}(C^\circ)$.
Note that $S\not\subset C^\circ$ in general
and $X^\circ\to C^\circ$ is not necessarily smooth. 
Then the commutative diagram
\begin{equation*}
\xymatrix{
&0\ar[d]\\
&\Omega^1_{C}(\log T)\ot F^1\cH_e\ar[d]\\
F^1\cH_e\ar[r]^{\nabla\phantom{AAAAA}}\ar[d]_=&\Omega^1_C(\log T)\ot\cH_e\ar[d]\\
F^1\cH_e\ar[r]^{\ol{\nabla}\phantom{AAAAAA}} &\Omega^1_C(\log T)\ot \Gr_F^0 \cH_e\ar[d]\\
&0}
\end{equation*}
induces an isomorphism
\begin{equation*}
\L^\circ:=\G(C^\circ,\Omega^1_C(\log T)\ot F^1\cH_e)\os{\simeq}{\to}
H^1(C^\circ,F^1\cH_e\to\Omega^1_C(\log T)\ot\cH_e). 
\end{equation*}
Note that $\L^\circ \subset \G(X^\circ, \O^2_X(\log Z))$. 

\begin{lem}
There are natural injections
$$F^1H^2_\dR(X)_0 \hookrightarrow F^1H^2_\dR(U)_0
\hookrightarrow \L^\circ. $$
\end{lem}
\begin{proof}
The first injectivity follows from Zariski's lemma (cf. \cite[III, (8.2)]{barth}). 
Since
$$H^2_\dR(U)_0\simeq H^1(S,\cH\to\Omega^1_S\ot\cH)
\simeq H^1(C,	\cH_e\to\Omega^1_C(\log T)\ot\cH_e)$$
and 
$$F^1H^1(S,\cH\to\Omega^1_S\ot\cH)
=H^1(C,F^1\cH_e\to\Omega^1_C(\log T)\ot\cH_e)$$ (cf. \cite[\S 5]{s-z}), 
the second injectivity follows from that of 
$F^1H^2_\dR(U)_0\to F^1H_\dR^2(U\cap X^\circ)_0$. 
\end{proof}

Define $\L(X) \subset \L(U) \subset \L^\circ$ to be the images of $F^1H^2_\dR(X)_0$, $F^1H_\dR^2(U)_0$, respectively. 
By the commutative diagram
$$\xymatrix{
F^1H^2_\dR (X)_0 \ar[d]_\simeq\ar[r]& F^1H^2_\dR (U)_0\ar[r]\ar[d]_\simeq \ar[r] &H_1^\dR(Z)\ar[d]\\
\L(X) \ar[r]& \L(U) \ar[r] & H^0(X^\circ,\O^2_X(\log Z)/\O_X^2)
}$$
we have $\L(X)\subset \G(X^\circ, \O^2_X)$. 
For any cohomology class $\o \in F^1H_\dR^2(X)_0$, let $\o^\circ \in \L(X)$ denote the corresponding rational $2$-form.  

\subsection{Main result}

Now, let $D \subset X^\circ$ be a finite union of fibers. 
We give a description of 
$$\r_{X,D}\colon H_1(D,\Q)\cap H^{0,0} \to \Coker\left[H_2(X,\Q) \to (F^1H_\dR^2(X)_0)^*\right],$$
the restriction to $F^1H_\dR^2(X)_0 \subset F^1H_\dR(X)_D$ of the map given in Sect. 7.1. 
Note that this factors through $\r_{X^\circ,D}$. 
We regard an element $\y \in \L^\circ$ as an element of $\cA^2(X^\circ)$. 
For the dimension reasons, we have $\wt{i}^*\y=0$ and $d\y=0$. 
Hence $(0,0,\y)$ as in \eqref{element} defines a cohomology class $\wh\y \in H_\dR^2(X^\circ,D)$. Note that $\wh\y$ does not necessarily belong to $F^1$. 
For any $\o\in F^1H_\dR^2(X)_0$, write $\wh\o$ instead of $\wh{\o^\circ}$. 

\begin{thm}\label{thm apdx}\ 
\begin{enumerate}
\item 
For any $\o\in F^1H_\dR^2(X)_0$, we have $\wh\o \in F^1H^2_\dR(X^\circ,D)_0$ and it lifts $\o|_{X^\circ}$. 
\item
For any $\g \in H_1(D,\Q) \cap H^{0,0}$,  choose $\G \in H_2(X^\circ,D)$ such that $\pd(\G)=\g$. 
Then we have 
$$\r_{X,D}(\g) = \left[\o \mapsto \int_\G \o^\circ\right].$$
\end{enumerate}
\end{thm}

\begin{proof}
By \eqref{exp-14}, the assertion (ii) follows immediately from (i). 
By Lemma \ref{lem-birat}, we may assume that $D_\mathrm{red}$ 
and $Z_\mathrm{red}$ are divisors with normal crossings.
It suffices to prove the case where $D=f^{-1}(P)$, $P \in C^\circ$. 

For a Zariski sheaf $\mathscr F$, let $(\check{C}^\bullet({\mathscr F}),\delta)$ denote its \v{C}ech complex.
Firstly, $H^2_\dR(X)$ is given by the cohomology in the middle of the complex
$$\check{C}^1(\cO_X)\times
\check{C}^0(\Omega^1_X)
\os{\cD_1}{\to}
\check{C}^2(\cO_X)\times
\check{C}^1(\Omega^1_X)
\times
\check{C}^0(\Omega^2_X)
\os{\cD_2}{\to}
\check{C}^3(\cO_X)\times
\check{C}^2(\Omega^1_X)
\times
\check{C}^1(\Omega^2_X).$$
A description of $H^2_\dR(U)=H^2(X,\Omega^\bullet_X(\log Z))$ is given similarly.
Finally, $H^2_\dR(X,D)$ is given by the complex
\begin{align*}
\check{C}^1(\cO_X)\times
\check{C}^0(\cO_{\wt{D}}\op\Omega^1_X)
&\os{\cD_3}{\to}
\check{C}^2(\cO_X)\times
\check{C}^1(\cO_{\wt{D}}\op\Omega^1_X)
\times
\check{C}^0(\cO_{\wt{\Sigma}}/\cO_\Sigma\op\Omega^1_{\wt{D}}\op\Omega^2_X)\\
&\os{\cD_4}{\to}
\check{C}^3(\cO_X)\times
\check{C}^2(\cO_{\wt{D}}\op\Omega^1_X)
\times
\check{C}^1(\cO_{\wt{\Sigma}}/\cO_\Sigma\op\Omega^1_{\wt{D}}\op\Omega^2_X). 
\end{align*}
Let $\o \in F^1H_\dR^2(X)_0$ and take its representative
$z=(0)\times (\a_{ij}) \times (\b_i) \in \Ker(\sD_2)$. 
Since $\o \in F^1H_\dR^2(X)_D$, there exists $(\epsilon_i)\in \check{C}^0(\Omega^1_{\wt{D}})$ such that
$\alpha_{ij}|_{\wt{D}}=\epsilon_j-\epsilon_i$. 
If we put
$$z_{X,D}=(0)\times (0,\alpha_{ij})\times(0,\epsilon_i,\beta_i), $$
then $z_{X,D} \in \Ker(\cD_4)$. 
By the definition of the Hodge filtration, it represents a class $\o_{X,D} \in F^1H_\dR^2(X,D)$ which lifts $\o$. Let $\o_{X,D}|_{X^\circ}$ be its image in $H_\dR^2(X^\circ,D)$. 

Let $\wh\o \in H_\dR^2(X^\circ,D)$ be the class of the \v{C}ech cocycle
$$\wh z:=(0)\times(0,0)\times(0,0,\o^\circ).$$
The group $H^1(C^\circ, F^1\sH_e \to \O^1_C(\log T)\ot \sH_e)$ is given by the complex
\begin{align*}
\check{C}^0(F^1\sH_e|_{C^\circ}) 
& \os{\sD_5}\to \check{C}^1(F^1\sH_e|_{C^\circ}) \times \check{C}^0(\O^1_C(\log T)\ot \sH_e|_{C^\circ})
\\&\os{\sD_6}\to 
\check{C}^2(F^1\sH_e|_{C^\circ}) \times \check{C}^1(\O^1_C(\log T)\ot \sH_e|_{C^\circ}).
\end{align*}
By the definition of $\o^\circ$, there exists 
$y= (\nu_i) \in \check{C}^0(F^1\cH_e|_{C^\circ})$ such that 
$\cD_5(y)=(\alpha_{ij})\times (\beta_i)-(0)\times (\omega^\circ)$, i.e. 
$\nu_j-\nu_i=\a_{ij}$, $d\nu_i=\b_i-\o^\circ$. 
Hence we have
$$z_{X,D}|_{X^\circ}- \wh z=
(0)\times(0,\nu_j-\nu_i) \times(0,\e_i,d\nu_i).$$
It is clear that this vanishes in $H_\dR^2(X^\circ)$, hence $\wh\o$ lifts $\o|_{X^\circ}$. 

We are left to show that the class of $\wh\o$ lies in $F^1$. 
Let $V$ be a sufficiently small neighborhood of $D$ so that we have an exact sequence 
$$0 \to\Omega^1_V\to
\Omega^1_V(\log D)\os{\Res}{\to} \wt i_*\cO_{\wt{D}}\to 0.$$
Since 
$H_\dR^2(X^\circ,D)/F^1 \to H_\dR^2(V,D)/F^1$ is injective, 
it suffices to show the claim after restricting to $V$. 
Since $\Res(\nu_j)-\Res(\nu_i)=\Res(\a_{ij})=0$, $(\Res(\nu_i))$ defines a class $e \in H^0(\wt{D},\sO_{\wt D})$. 
Consider the composite
$$H^0(\wt{D},\sO_{\wt D}) \os\d\to H^1(V,\O^1_V) \os{\wt i^*}\to H^1(\wt D,\O^1_{\wt D})
 \simeq H^2_\dR(\wt D)$$
where $\d$ is the connecting map. 
Then $(\wt i^*\circ \d)(e)$ is represented by 
$(\a_{ij}|_{\wt D}) \in \check{C}^1(\O_{\wt D})$. 
Therefore, under the above isomorphism, 
$(\wt i^*\circ \d)(e)$ corresponds to $\wt i^*(\o)=0$. 
Let $t\in \cO_{C,P}$ be a uniformizer at $P$. 
By Zariski's lemma (cf. \cite[III, (8.2)]{barth}), $\Ker(\wt i^*\circ \d)$ is one-dimensional and generated by $\Res(\frac{dt}{t})$. Hence there exists a constant $c$ such that 
$$\theta_i:=\nu_i-c \frac{dt}{t}$$
has no pole along $D$. 
By replacing $\nu_i$ with $\theta_i$ and taking $\ve_i=\theta_i|_{\wt D}$, 
we see that 
$\o_{X,D}|_V-\wh\o|_V$ is in the image of $F^1H_\dR^1(V) \to H_\dR^2(V,D)$. 
Hence we obtain $\wh\o \in F^1$ and the proof is complete.
\end{proof}


\end{document}